\documentclass[12pt,a4paper,reqno]{amsart}
\usepackage{amsmath,amsthm,amssymb}
\usepackage{graphicx} 
\usepackage[noadjust]{cite}
\usepackage{amsthm}
\usepackage{bm}
\usepackage[top=29truemm,bottom=29truemm,left=23truemm,right=23truemm]{geometry} 
\usepackage{url}
\usepackage{mathtools}
\usepackage{mathdots}
\usepackage{enumitem} 
\usepackage{comment}

\usepackage{float}

\numberwithin{equation}{section}

\theoremstyle{plain}
\newtheorem{thm}{Theorem}[section]
\newtheorem{lem}[thm]{Lemma}
\newtheorem{prop}[thm]{Proposition}
\newtheorem{cor}[thm]{Corollary}

\theoremstyle{definition}

\theoremstyle{remark}
\newtheorem{rem}[thm]{Remark}

\begin{document}
\title[Lorentzian homogeneous Ricci-flat metrics]{Lorentzian homogeneous Ricci-flat metrics  \\ 
on almost abelian Lie groups}
\author{Yuichiro Sato \and Takanao Tsuyuki}

\date{\today}

\address{Global Education Center, Waseda University, 
Nishi-waseda, 1-6-1, Shinjuku-ku, Tokyo, 169-0805, Japan}

\email{yuichiro-sato@aoni.waseda.jp}

\address{Faculty of Informatics, Hokkaido Information University, 
Nishi-nopporo 59-2, Ebetsu, Hokkaido, 069-8585, Japan}

\email{tsuyuki@do-johodai.ac.jp}



\begin{abstract}
When the identity component of the full isometry group of a four-dimensional spacetime acts simply transitively, the unique Ricci-flat metric is the Petrov solution.
This isometry group is almost abelian; that is, its Lie algebra contains an abelian ideal of codimension one.
In this paper, we study Lorentzian left-invariant metrics on almost abelian Lie groups of dimension four or higher. 
In particular, we construct a Ricci-flat but non-flat metric that generalizes the Petrov solution to arbitrarily high dimensions.
The generalized solution is geodesically complete and admits closed timelike curves. 
The construction of the closed timelike curves is new even in the four-dimensional Petrov solution, as it requires no identification of coordinates.
\end{abstract}

\maketitle

\section{Introduction}
A homogeneous Ricci-flat Riemannian manifold is necessarily flat \cite{MR402650}.
Even for the pseudo-Riemannian case, Ricci-flat manifolds of dimension three or lower are flat.
Thus, homogeneous, Ricci-flat, and non-flat metrics can exist only for pseudo-Riemannian manifolds of dimension four or higher.

In the theory of general relativity, the metrics on vacuum spacetimes (without the cosmological constant) are Ricci-flat.
Since our spacetime is a Lorentzian manifold of dimension at least four, a homogeneous, Ricci-flat, and non-flat vacuum spacetime can indeed exist.
For the four-dimensional case, the following theorem is known:

\begin{thm}[{\cite[Theorem~12.2]{MR2003646}}]
    The only Ricci-flat metric admitting a simply-transitive four-dimensional Lie group as its maximal isometry group is given by
    \begin{align}
        k^2 ds^2 =e^{x_4}
        \left[ 
        \cos(\sqrt{3} x_{4})(-dx_{1}^{2} + dx_{2}^{2}) 
        - 2\sin(\sqrt{3} x_4)dx_{1} dx_{2} 
        \right] 
        + e^{-2x_4} dx_3^{2} + dx_4^{2},
        \label{Petrov_sol}
    \end{align}
    where $k$ is an arbitrary constant.
\end{thm}
This metric is found in Ref.~\cite{MR0164700} and is called the Petrov solution \cite{MR2429726}.
The main purpose of this paper is to generalize this solution to arbitrary dimensions.
The observed spacetime dimension is four, but higher dimensions are necessary for some theories, such as string/M theories \cite{MR3155203,MR1334520}, and are also explored experimentally \cite{ParticleDataGroup:2024cfk}. 

The Killing vectors of the Petrov solution are \cite{MR2003646}
\begin{align*}
    \partial_{x_1},\ \partial_{x_2},\ \partial_{x_3},\ \partial_{x_4}+x_3\partial_{x_3}+\frac12(\sqrt{3}x_1-x_2)\partial_{x_2}-\frac12(x_1+\sqrt{3}x_2)\partial_{x_1}. 
\end{align*}
The first three vectors commute with each other, and the commutators with the fourth can be written as a linear combination of the first three vectors.
In other words, the Lie algebra generated by the Killing vectors has a codimension-one abelian ideal; such a Lie algebra is called almost abelian \cite{MR2995204,MR3049631,MR3772583,MR4533011}.\footnote{The term ``almost abelian'' is sometimes used with different meanings. For example, the almost abelian Lie algebra in Ref.~\cite{MR195916} corresponds to a special case $A=I_{n-1}$ in our notation~\eqref{eA}. Ref.~\cite{MR4462443} excludes abelian algebras from almost abelian algebras.}
A connected Lie group is almost abelian if its Lie algebra is almost abelian.
Almost abelian Lie groups are metabelian, and metabelian Lie groups are equivalent to $2$-step solvable ones. 
Almost abelian Lie groups can be nilpotent (e.g., the three-dimensional Heisenberg group), but 
this is not necessarily the case (e.g., the isometry group of the Petrov solution).
Ricci-flat Lorentzian homogeneous metrics have primarily been studied in the context of four-dimensional groups \cite{MR0164700,MR0524082,MR0475635,MR2003646,MR3223494,MR2429726} or nilpotent groups \cite{MR3263659,MR3954004,MR4101481,MR4319913,MR4609835}.
This work differs from previous studies because we focus on Lie groups that are of arbitrary dimension and not necessarily nilpotent.

In this paper, we show the explicit conditions under which a left-invariant Lorentzian metric on an almost abelian Lie group becomes Ricci-flat and non-flat. 
In particular, we generalize the Petrov solution to arbitrarily higher dimensions.
The main result of this paper is as follows:

\begin{thm} \label{th}
    Let $n\geq 4$, and let $G$ be a simply connected $n$-dimensional almost abelian Lie group. 
    Suppose that $G$ acts simply transitively on a Lorentzian manifold $(G,ds^2)$, and that $G$ is the identity component of the full isometry group of $(G,ds^2)$. 
    Then, up to isometry, the only Ricci-flat metric is
\begin{equation}
ds^{2} = e^{-2\alpha x_{n}}
\left[ 
\cos(2\beta x_{n})(-dx_{1}^{2} + dx_{2}^{2}) 
- 2\sin(2\beta x_{n})dx_{1} dx_{2} 
\right] 
+ \sum_{i = 3}^{n-1} e^{-2\lambda_{i} x_{n}} dx_{i}^{2} + dx_{n}^{2}, 
\label{main}
\end{equation}
where $\alpha, \beta, \lambda_{3}, \ldots, \lambda_{n-1}$ satisfy 
\begin{align}
&\alpha \le0,\quad \beta > 0, \notag \\
&\lambda_{3} > \ldots > \lambda_{n-1}, \label{simp} \\
&2\alpha + \sum_{i=3}^{n-1} \lambda_{i} = 
2\alpha^{2} - 2\beta^{2} + \sum_{i=3}^{n-1} \lambda_{i}^{2} = 0, \label{Ric_flat_cond}
\end{align}
and $x_{1}, \ldots, x_{n}$ 
denote the canonical coordinates of $\mathbb{R}^{n}$. 
\end{thm}

Since a simply connected almost abelian Lie group is diffeomorphic to \(\mathbb{R}^n\), 
the underlying manifold is automatically diffeomorphic to \(\mathbb{R}^n\) whenever such a group acts simply transitively. 
This solution reduces to the Petrov solution \eqref{Petrov_sol} in the case $n=4$ (see Corollary~\ref{pl3}). 

After proving the main theorem, we also show that the solution is geodesically complete. 
That is, every geodesic can be extended to a geodesic defined for all real values of its parameter. 
We also study the causal structure of the solution. We find it to be totally vicious; that is, every point lies on a closed timelike curve (CTC). 
An explicit CTC is constructed without any coordinate identification (see Proposition~\ref{causality}). 
This CTC construction is new even for the four-dimensional case.
Similarly to the Petrov solution, the G{\"o}del solution \cite{MR31841} is homogeneous and admits CTCs. 
Its main differences from the Petrov solution are that it is not a vacuum solution and the isometry group is multiply transitive \cite{MR2003646}. 

This paper is organized as follows. 
In Section~\ref{smet}, we introduce almost abelian Lie groups and left-invariant Lorentzian metrics on them.
In Section~\ref{sric}, we derive the necessary and sufficient conditions for such metrics to be Ricci-flat, as well as conditions for them to be flat.
In Section~\ref{Petrov}, we examine whether the identity component of the full isometry groups of the Ricci-flat metrics is simply transitive and complete the proof of the main theorem.
In Section~\ref{property}, we investigate properties of the generalized Petrov solution, such as geodesic completeness and causal structure. 

\section{Almost abelian Lie groups} \label{smet}

Let $G$ be an $n$-dimensional Lie group, $\mathfrak{g}$ its Lie algebra,
and $\{X_{a}\}_{1 \leq a \leq n}$ left-invariant vector fields on $G$.
Then we define the structure constants ${C^a}_{bc}$ of $\mathfrak{g}$ as
\begin{equation*}
[X_{b}, X_{c}] = \sum_{a=1}^{n} {C^a}_{bc} X_{a}. \label{estr}
\end{equation*}
In addition, we define the left-invariant $1$-forms $\{\omega^{a}\}_{1 \leq a \leq n}$ on $G$ as $\omega^{a}(X_{b}) = \delta^{a}_{b}$. 
A left-invariant metric on $G$ can be written as
\begin{align}
    ds^2 = \sum_{a, b = 1}^{n} \eta_{ab}\omega^a\omega^b, \label{eds}
\end{align}
where $\eta_{ab}$ are constants.

In this paper, we focus on the case where $G$ is an almost abelian Lie group. 
Let $\mathfrak{g}$ be an $n$-dimensional real Lie algebra. 
It is \textit{almost abelian} if it contains an $(n-1)$-dimensional abelian ideal $\mathfrak{a}$ \cite{MR4533011}.
A connected $n$-dimensional Lie group $G$ is \textit{almost abelian} if its Lie algebra $\mathfrak{g}$ is almost abelian.
In this setting, the left-invariant vectors are pairwise commuting except for one, which we denote $X_n$. The structure constants become
\begin{align}
    {C^{i}}_{jk}={C^{n}}_{jk}={C^{n}}_{nk}=0,
\end{align}
where the indices $i,j,k$ run over $1,\dots,n-1$.
We define the possibly non-zero components ${C^{i}}_{nj}$ to be components of the associated matrix $A = [A^{i}_{j}]_{1 \leq i, j \leq n-1}$ 
of $\mathfrak{g}$ with respect to 
$\{ X_{a} \}_{1 \leq a \leq n}$;
\begin{align}
[X_n,X_j] = \sum_{i = 1}^{n-1} A^{i}_{j}X_i \quad (1 \leq j \leq n-1). \label{eA} 
\end{align}

Let $\mathrm{M}_{m}\mathbb{R}$ denote 
the set of $m \times m$ real square matrices. 
The Jacobi identity holds for arbitrary 
$A \in \mathrm{M}_{n-1}\mathbb{R}$. 
For example,
\begin{align}
    \sum_{a=1}^{n} \left( {C^i}_{j a}{C^a}_{n n}+{C^i}_{n a}
    {C^a}_{n j}+
    {C^i}_{n a}{C^a}_{j n} \right) 
    = \sum_{k=1}^{n-1} \left( A^{i}_{k} A^{k}_{j} -A^{i}_{k} A^{k}_{j} \right) 
    = 0, \label{eqC}
\end{align}
and the others also identically vanish.

In this paper, we denote a left-invariant Lorentzian metric 
by $ds^{2}$ on an almost abelian Lie group $G$.
Then the pair $(G, ds^{2})$ is a homogeneous Lorentzian manifold, 
and the given metric corresponds to 
the non-degenerate symmetric bilinear form
$\langle \, , \, \rangle$ with Lorentzian signature on the Lie algebra $\mathfrak{g}$.
In addition, $\mathfrak{a} \subset \mathfrak{g}$ denotes
a codimension-one abelian ideal of the almost abelian Lie algebra $\mathfrak{g}$.

\begin{lem} \label{metric}
Let $G$ be an $n$-dimensional almost abelian Lie group,
$ds^{2}$ a left-invariant Lorentzian metric on $G$.
Then there exists a basis 
$\{X_{1}, \ldots, X_{n-1}, X_{n}\}$ of $\mathfrak{g}$ such that
\begin{equation*}
{\rm span}_{\mathbb{R}}\{ X_{1}, \ldots, X_{n-1} \} = \mathfrak{a}, \quad
\mathfrak{g} = \mathfrak{a} \rtimes \mathbb{R} X_{n}
\end{equation*}
and the bilinear form $\langle \, , \, \rangle$ on $\mathfrak{g}$ can be expressed by one of the following matrices:
\begin{equation}
({\rm a}) :
\begin{bmatrix}
\, 1 & & & \, \, \\
\,  & \ddots & & \, \, \\
\,  & & 1 & \, \, \\
\,  & & & -1 \, \,
\end{bmatrix}
, \quad
({\rm b}) :
\begin{bmatrix}
\, -1 & & & \, \, \\
\, & 1 & & \, \, \\
\, & & \ddots & \, \, \\
\, & & & 1 \, \,
\end{bmatrix}
, \quad
({\rm c}) :
\begin{bmatrix}
\, 1 & & & & \, \, \\
\, & \ddots & & & \, \, \\
\, & & 1 & & \, \, \\
\, & & & 0 & 1 \, \, \\
\, & & & 1 & 0 \, \, \\
\end{bmatrix}
. \label{eq:eta}
\end{equation}
\end{lem}

\begin{proof}
Since there exists a codimension-one abelian ideal
$\mathfrak{a} \subset \mathfrak{g}$,
the induced symmetric bilinear form on $\mathfrak{a}$ is either Riemannian, Lorentzian, or degenerate.
In the Riemannian case,
we can take an orthonormal basis $\{X_{1}, \ldots, X_{n-1}\}$ of $\mathfrak{a}$.
Therefore, it suffices to choose a unit timelike vector $X_{n}$ of the orthogonal complement $\mathfrak{a}^{\bot}$.
In the Lorentzian case,
we can take an orthonormal basis $\{X_{1}, \ldots, X_{n-1}\}$ of $\mathfrak{a}$ in the following sense;
\begin{equation*}
\langle X_{1}, X_{1} \rangle = -1, \quad
\langle X_{i}, X_{i} \rangle = 1 \quad (2 \leq i \leq n-1).
\end{equation*}
Similarly, 
it suffices to choose a unit spacelike vector $X_{n}$ in the orthogonal complement $\mathfrak{a}^{\bot}$.
In the degenerate case,
since the bilinear form on $\mathfrak{g}$ is non-degenerate,
we can take a quasi-orthonormal basis of $\mathfrak{g}$ such that
\begin{equation*}
\langle X_{j}, X_{j} \rangle = \langle X_{n-1}, X_{n} \rangle = 1
\quad (1 \leq j \leq n-2),
\end{equation*}
the others are zero,
and ${\rm span}_{\mathbb{R}}\{ X_{1}, \ldots, X_{n-1} \} = \mathfrak{a}$.
\end{proof}

For a given almost abelian Lie group $(G, ds^{2})$ 
with a left-invariant Lorentzian metric, 
it suffices to deal with $(\mathfrak{g}, \langle \, , \, \rangle)$ 
instead of $(G, ds^{2})$. 
We call it the Lorentzian almost abelian Lie algebra for $(G, ds^{2})$.

For any $X \in \mathfrak{g}$, 
we define the operator $L_{X} : \mathfrak{g} \rightarrow \mathfrak{g}$ as
\begin{equation}
2\langle L_X(Y), Z \rangle
= -\langle X, [Y, Z] \rangle + \langle Y, [Z, X] \rangle + \langle Z, [X, Y] \rangle,  \label{left_op}
\end{equation}
where $Y, Z \in \mathfrak{g}$.
Then the curvature operator is defined by 
\[ R(X, Y) := [L_{X}, L_{Y}] - L_{[X, Y]} \quad (X, Y \in \mathfrak{g}). \]
The Ricci tensor 
${\rm Ric} : \mathfrak{g} \times \mathfrak{g} \rightarrow \mathbb{R}$
of $(\mathfrak{g}, \langle \, , \, \rangle)$ can be computed using the following formula. 

\begin{prop}[{\protect\cite[Corollary~7.38]{MR2371700},
\protect\cite[Lemma~2.1]{MR4752280}}] 
\label{Ric_curv}
Let $\{ e_{a} \}_{1 \leq a \leq n}$ be an orthonormal basis of 
$(\mathfrak{g}, \langle \, , \, \rangle )$. Then 
\begin{align*}
{\rm Ric}(X, Y) &= -\frac{1}{2} \sum_{a=1}^{n} 
\langle [X, e_{a}], [Y, e_{a}] \rangle \varepsilon_{a} + 
\frac{1}{4}\sum_{a,b=1}^{n} \langle X, [e_{a}, e_{b}] \rangle 
\langle Y, [e_{a},e_{b}] \rangle \varepsilon_{a}\varepsilon_{b} \\
&\qquad \quad -\frac{1}{2} K(X,Y) 
- \frac{1}{2}(\langle [H,X], Y \rangle + \langle [H, Y], X \rangle), 
\end{align*}
where $X, Y \in \mathfrak{g}$, 
$\varepsilon_{a} = \langle e_{a}, e_{a} \rangle = \pm 1$, 
and $K$ is the Killing form of $\mathfrak{g}$ defined by 
\[ K(X, Y) := 
\mathrm{tr} \, \left[ (\mathrm{ad} \, X) \, (\mathrm{ad} \, Y) \right],\]
and $H$ is the mean curvature vector of 
$(\mathfrak{g}, \langle \, , \, \rangle)$ defined by 
\[ \langle H, X \rangle := \mathrm{tr} \, (\mathrm{ad} \, X). \]
\end{prop} 

The Lorentzian almost abelian Lie algebra 
$(\mathfrak{g}, \langle \, , \, \rangle)$ for $(G, ds^{2})$ 
is called \textit{flat} if $R = 0$ holds, 
and it is called \textit{Ricci-flat} if ${\rm Ric} = 0$ holds. 
The definitions correspond to the flatness and Ricci-flatness of 
$(G, ds^{2})$, respectively. 

\begin{lem} \label{Kill} 
Let $\mathfrak{g}$ be an $n$-dimensional almost abelian Lie group with 
the associated matrix $A \in \mathrm{M}_{n-1}\mathbb{R}$ 
with respect to $\{X_{1}, \ldots, X_{n}\}$. 
Then the Killing form $K$ is given by 
\[ K(X_{i}, X_{j}) = K(X_{n}, X_{j}) = 0, \quad 
K(X_{n}, X_{n} ) = \mathrm{tr} \, A^{2} \]
for $i, j = 1, \ldots, n-1$.
\end{lem} 

\begin{proof}
By direct computation, we have 
\begin{equation} 
\mathrm{ad} \, X_{i} = 
\begin{bmatrix}
\, O & -\bm{a}_{i} \, \\
\, \bm{0}^{T} & 0 \, \\ 
\end{bmatrix}
, \quad \mathrm{ad} \, X_{n} = 
\begin{bmatrix}
\, A & \bm{0} \, \, \\
\, \bm{0}^{T} & 0 \, \, \\
\end{bmatrix}
, \label{ad_rep}
\end{equation}
where $i = 1, \ldots, n-1$ and $A= [\bm{a}_{1}, \ldots, \bm{a}_{n-1}] 
\in \mathrm{M}_{n-1}\mathbb{R}$.
From the definition of the Killing form, we have the claim. 
\end{proof}

\begin{lem} \label{mean_curv} 
Corresponding to three metrics (a), (b) and (c) in Lemma~\ref{metric}, 
the mean curvature vector $H \in \mathfrak{g}$ 
of the Lorentzian almost abelian Lie algebra 
$(\mathfrak{g}, \langle \, , \, \rangle)$ coincides with 
\[ ({\mathrm a}) :-(\mathrm{tr}\, A)X_{n}, \quad 
({\mathrm b}) :(\mathrm{tr}\, A)X_{n}, \quad 
({\mathrm c}) :(\mathrm{tr}\, A)X_{n-1}, \]
respectively.
\end{lem}

\begin{proof}
It immediately follows from the formulae~(\ref{ad_rep}). 
\end{proof}

When an $n$-dimensional almost abelian Lie group $G$ is simply connected, 
$G$ is isomorphic to 
$\mathbb{R}^{n}$ with the suitable Lie group structure \cite{MR4533011}. 
By identifying them and using the coordinates of $\mathbb{R}^{n}$, we can express the left-invariant Lorentzian metric as follows.
\begin{prop}
The left-invariant Lorentzian metric of an almost abelian Lie group whose associated matrix is $A$ can be written as
\begin{align}
    ds^2 = d\bm{x}^T
    \begin{bmatrix}
        e^{-A^Tx_n} & \bm{0}\\
        \bm{0}^T & 1
    \end{bmatrix}
    \eta
    \begin{bmatrix}
        e^{-Ax_n} & \bm{0}\\
        \bm{0}^T & 1
    \end{bmatrix}
    d\bm{x}, \label{emet}
\end{align}
where $d\bm{x} = (dx_1,\dots,dx_n)^{T}$ and $\eta$ is one of the matrices in \eqref{eq:eta}.
\end{prop}
\begin{proof}
    From \cite[Corollary~2.6]{MR4859975}, the left-invariant vectors $\{ X_a \}_{1 \leq a \leq n}$ and $1$-forms 
$\{ \omega^{a} \}_{1 \leq a \leq n}$ can be written in the coordinate basis 
as
\begin{align}
    X_i &= \sum_{j=1}^{n-1}\left(e^{A x_n}\right)_{i}^j \partial_{x_j},\quad X_n=\partial_{x_n}, \\
    \omega^i &= \sum_{j=1}^{n-1}\left(e^{-A x_n}\right)_{j}^i dx_j,\quad \omega^n=dx_n. \label{eomega}
\end{align}
By substituting \eqref{eomega} into \eqref{eds}, we have the claim.
\end{proof}

\section{Ricci-flat and flat conditions} \label{sric}

We set the three metrics in \eqref{eq:eta} as 
\[ \mathrm{(a)} : \langle \, , \, \rangle_{\mathrm{a}}, \quad 
\mathrm{(b)} : \langle \, , \, \rangle_{\mathrm{b}}, \quad 
\mathrm{(c)} : \langle \, , \, \rangle_{\mathrm{c}}. \]
Moreover, in response to the above metrics, 
we decompose the associated matrix $A$ as 
\begin{equation}
\mathrm{(a)} : S + T, \quad 
\mathrm{(b)} : S_{L} + T_{L}, \quad 
\mathrm{(c)} : 
\begin{bmatrix}
    A' & \bm{b} \, \\
    \bm{c}^{T} & d \, \\  
\end{bmatrix} 
, \label{split}
\end{equation}
where $S, \, T$ are the symmetric and skew-symmetric matrices, 
$S_{L}, \, T_{L}$ satisfy 
\[ S_{L}^{T} = JS_{L}J, \quad T_{L}^{T} = -JT_{L}J 
\quad (J := \mathrm{Diag}(-1,1, \ldots, 1) 
\in \mathrm{M}_{n-1}\mathbb{R}), \] and $A' \in \mathrm{M}_{n-2}\mathbb{R}, \, 
\bm{b}, \bm{c} \in \mathbb{R}^{n-2}, \, d \in \mathbb{R}$. 
Furthermore, we decompose $A'$ into symmetric and skew-symmetric parts 
as $A' = S' + T'$ in case (c). 

As a remark, in case (b), 
the decomposition $A = S_{L} + T_{L}$ is uniquely determined. 
We call it the \textit{$J$-decomposition} of $A$. 

\subsection{Ricci-flat conditions}
We determine, up to isometry, the necessary and sufficient conditions for almost abelian Lie groups to be Ricci-flat. 

\begin{prop} \label{Ric_tensor}
For the Lorentzian almost abelian Lie algebra $(\mathfrak{g}, \langle \, , \, \rangle)$ 
where $\langle \, , \, \rangle$ is one of the three cases in Lemma~\ref{metric},
the Ricci tensor $\mathrm{Ric} 
= \left[ \mathrm{Ric}(X_{a}, X_{b}) \right]
_{1 \leq a, b \leq n}$ is given by 
\begin{align}
 &(\rm{a}) : 
\begin{bmatrix}
    \, \, [S, T] + (\mathrm{tr}\, S)S & \bm{0} \, \\
    \, \, \bm{0}^{T} & -\mathrm{tr}\, S^{2} \,  
\end{bmatrix}
, \label{Ric_a}
\\
&(\rm{b}) : 
\begin{bmatrix}
    \, -J([S_{L}, T_{L}] + (\mathrm{tr}\, S_{L})S_{L}) & \bm{0} \, \\
    \, \bm{0}^{T} & -\mathrm{tr}\,S_{L}^{2} \, \\
\end{bmatrix}
, \label{Ric_b}\\
&(\rm{c}) : 
\frac{1}{2} 
\begin{bmatrix}
    \, -\bm{b}\bm{b}^{T} & \bm{0} 
    & [(\mathrm{tr}\, S')I_{n-2} + S' + T']\bm{b}  \, \\ 
    \, \bm{0}^{T} & 0 & ||\bm{b}||^{2} \, \\
    \, \bm{b}^{T}[(\mathrm{tr}\, S')I_{n-2} + S' - T'] 
    & ||\bm{b}||^{2} & 2d(\mathrm{tr}\, S') - 2\mathrm{tr}\, S'\,\!^{2}
    - 2\bm{b}^{T}\bm{c} \, \\
\end{bmatrix}
, \label{Ric_c}
\end{align}
where $|| \cdot ||$ denotes the Euclidean norm of the Euclidean space. 
\end{prop} 

\begin{proof}
We use Proposition~\ref{Ric_curv}. 
In cases (a) and (b), $\{X_{1}, \ldots, X_{n}\}$ is an orthonormal basis, 
so that we can apply the formula directly. 
In case (c), since the generators are not an orthonormal basis, 
we need to exchange the generators such that
\[ e_{i} = X_{i} \ (1 \leq i \leq n-2), \quad 
e_{n-1} = \frac{1}{\sqrt{2}} (X_{n-1} + X_{n}), \quad 
e_{n} = \frac{1}{\sqrt{2}} (X_{n-1} - X_{n}). \]
Then $\{e_{1}, \ldots, e_{n-1}, e_{n} \}$ is an orthonormal basis, 
and we can apply Proposition~\ref{Ric_curv}. 
We obtain the result by direct computation.
The computation is shown explicitly in the Appendix~\ref{Ric_op}.
\end{proof} 


Here, we check the transformation rules for replacing the generators. 
Let $\mathrm{GL}_{m}\mathbb{R}$ denote the set of invertible matrices 
of $\mathrm{M}_{m}\mathbb{R}$. 

\begin{lem} \label{trans} 
When we transform the generators as
\begin{equation*}
[\overline{X}_{1}, \ldots, \overline{X}_{n-1}, \overline{X}_{n}] = 
[X_{1}, \ldots, X_{n-1}, X_{n}] 
\begin{bmatrix}
\, P & \bm{0} \, \, \\
\, \bm{0}^{T} & c \, \, \\
\end{bmatrix}
\quad (P \in \mathrm{GL}_{n-1}\mathbb{R}, \ c \neq 0),
\end{equation*}
the associated matrix $A \in \mathrm{M}_{n-1}\mathbb{R}$ and the metric matrix $\eta$ are transformed to $\overline{A}$ and $\overline{\eta}$ as 
\begin{equation*}
\overline{A} = c P^{-1} A P, \quad \overline{\eta} = 
\begin{bmatrix}
\, P^{T} & \bm{0} \, \, \\
\, \bm{0}^{T} & c \, \, \\
\end{bmatrix}
\eta
\begin{bmatrix}
\, P & \bm{0} \, \, \\
\, \bm{0}^{T} & c \, \, \\
\end{bmatrix}
. 
\end{equation*}
\end{lem}

\begin{proof}
The latter transformation formula is well-known. 
By using 
\[ \mathrm{span}_{\mathbb{R}}\{ X_{1}, \ldots, X_{n-1} \} = \mathrm{span}_{\mathbb{R}}\{ \overline{X}_{1}, \ldots, \overline{X}_{n-1} \} 
= \mathfrak{a}, \quad \overline{X}_{n} = c X_{n}, \]
we have the former transformation formula by direct computations. 
\end{proof}

\begin{cor}
    The metric \eqref{emet} is Ricci-flat if and only if the associated matrix $A$ under the decomposition \eqref{split} satisfies
    \begin{align}
        &(\mathrm{a}) : S=O.\\
        &(\mathrm{b}) : \mathrm{tr}\, S_{L} = \mathrm{tr}\, S_{L}^{2} = 0,\ T_L=O.
        \label{eq:trsl}\\
        &(\mathrm{c}) : S' = O \ (\mathrm{if} \ d = 0), \ 
\mathrm{tr}\, S'\,\!^{2} = \mathrm{tr}\, S' \ (\mathrm{if} \ d=1),\ \bm{b} = \bm{0}.  \label{eq:ric0c}
    \end{align}
\end{cor}

\begin{proof}
In case (a),  
from Proposition~\ref{Ric_tensor}, the Ricci-flat conditions are: 
\[ [S, T] + (\mathrm{tr}\, S)S = O, \quad 
\mathrm{tr}\, S^{2} = 0, \]
which is equivalent to $S = O$.
This is because $S$ is symmetric and $\mathrm{tr}\, S^{2} = ||S||^{2} = 0$, 
where $||S||$ denotes the matrix norm of $S$.  

In case (b), 
from Proposition~\ref{Ric_tensor}, the Ricci-flat conditions are: 
\[ [S_{L}, T_{L}] + (\mathrm{tr}\, S_{L})S_{L} = O, \quad 
\mathrm{tr}\, S_{L}^{2} = 0. \]
By taking the trace of the first equation, we have $\mathrm{tr}\,S_{L} = 0$.
The first condition becomes 
$[S_{L}, T_{L}]=O$.
In this case, the metric \eqref{emet} becomes
\begin{align*}
ds^{2} 
    = 
    d\bm{x}^T
    \begin{bmatrix}
        J e^{-2S_{L}x_n} & \bm{0}\\
        \bm{0}^T & 1
    \end{bmatrix}
    d\bm{x}.
\label{eq:ds2b}
\end{align*}
It is independent from $T_L$.
As long as we consider the isometric class of the Ricci-flat metrics, we may set $T_{L}=O$ without loss of generality.
Thus, we have the conditions \eqref{eq:trsl}.

In case (c), 
when we set 
\begin{equation*}
P = 
\begin{bmatrix}
\, 1 & & & \, \\
\, & \ddots & & \, \\
\, & & 1 & \, \\ 
\, & & & d \, \\
\end{bmatrix}
, \quad c = d^{-1} 
\quad (d \neq 0),
\end{equation*}
from Lemma~\ref{trans}, the transformation preserves the metric $\langle \, , \, \rangle_{\mathrm{c}}$ 
and the associated matrix $A$ is multiplied by $d^{-1}$. 
Actually, when $d \neq 0$ and $A$ is decomposed as \eqref{split}, by the transformation formula in Lemma~\ref{trans}, we have 
\[ d^{-1} 
\begin{bmatrix}
\, I_{n-2} & \bm{0} \, \\
\, \bm{0}^{T} & d^{-1} \, \\
\end{bmatrix}
\begin{bmatrix}
\, A' & \bm{b} \, \\
\, \bm{c}^{T} & d \, \\  
\end{bmatrix} 
\begin{bmatrix}
\, I_{n-2} & \bm{0} \, \\
\, \bm{0}^{T} & d \, \\
\end{bmatrix}
= 
\begin{bmatrix}
\, d^{-1} A' & \bm{b} \, \\
\, d^{-2}\bm{c}^{T} & 1 \, \\
\end{bmatrix}
.
\]
Thus, we can fix $d=0$ or $d=1$ without loss of generality. 
From Proposition~\ref{Ric_tensor}, the Ricci-flat conditions are: 
\[ ||\bm{b}|| = 0, \quad 
[(\mathrm{tr}\, S')I_{n-2} + S' + T']\bm{b} = \bm{0}, \quad 
2d(\mathrm{tr}\, S') - 2\mathrm{tr}\, S'\,\!^{2} - 2\bm{b}^{T}\bm{c} = 0. \]
The first condition immediately yields $\bm{b} = \bm{0}$, which makes the second condition trivial.
The third condition is equivalent to 
$\mathrm{tr}\, S'\,\!^{2} = d\mathrm{tr}\, S'$. 
It implies $S'=O$ if $d=0$.
Thus, we obtain the Ricci-flat conditions \eqref{eq:ric0c}.    
\end{proof}

We define the set of Lorentz transformations as 
\[ \mathrm{O}(1, m-1) 
= \{ P \in \mathrm{GL}_{m}\mathbb{R} \mid P^{T} J P = J \} \quad 
(J = \mathrm{Diag}(-1, 1, \ldots, 1) \in \mathrm{M}_{m}\mathbb{R}). \]
In case (b), we set $\overline{A} = P^{-1}AP$ for $P \in \mathrm{O}(1, n-2)$ from Lemma~\ref{trans}. 
Let $A = S_{L} + T_{L}$ and 
$\overline{A} = \overline{S}_{L} + \overline{T}_{L}$ be the $J$-decompositions. 
Then we have $\overline{S}_{L} = P^{-1}S_{L}P, \, \overline{T}_{L} = P^{-1}T_{L}P$ 
from the uniqueness of the $J$-decompositions. 
If $A$ satisfies the Ricci-flat conditions, $\overline{A}$ does trivially. 
Thus, we may replace $A$ with the normal form based on the action of Lorentz transformations. The following is known. 

\begin{prop}[{\cite[Corollary~2.4]{MR1332936}}] \label{normal_form}
Let $S_L \in \mathrm{M}_{m}\mathbb{R}$ satisfy $S_L^{T} = JS_LJ$. 
Then $S_L$ is conjugate by $P \in \mathrm{O}(1, m-1)$ to a diagonal matrix, 
or a direct sum of a diagonal matrix and one of the following matrices: 
\begin{align*}
({\rm i}) : 
\begin{bmatrix} 
\, \, \alpha & -\beta \, \\
\, \, \beta & \alpha \, \\
\end{bmatrix}
, \quad 
({\rm ii}) : 
\begin{bmatrix}
\, -1 + \gamma & -1 \, \\
\, 1 & 1 + \gamma \,
\\
\end{bmatrix}
, \quad 
({\rm iii}) : 
\begin{bmatrix}
\, -\delta & -\delta & -1 \, \, \\
\, \delta & \delta & 1 \, \, \\
\, 1 & 1 & 0 \, \\
\end{bmatrix}
\end{align*}
up to the sign. 
Here $\alpha, \, \beta, \gamma, \delta \in \mathbb{R}$ and $\beta > 0$.
\end{prop}

\begin{rem}  
The overall sign of the matrices in Proposition~\ref{normal_form} can be chosen to be positive by the transformation  
\[ P = I_{n-1}, \quad c = -1 \]
in Lemma~\ref{trans}.
In addition, we can choose $\alpha\le 0$ by the transformation 
\[ P = \mathrm{Diag}(-1, 1, \ldots, 1), \quad c = -1. \]
\end{rem}

\subsection{Flat conditions} \label{flat_cond}
Since any flat metric is trivially Ricci-flat, we are interested in metrics that are Ricci-flat but non-flat. Therefore, we establish the conditions for flatness in this subsection.

\begin{prop} \label{flat}
The metric \eqref{emet} is flat if and only if 
\begin{enumerate}
  \item[\rm{(a)} :] Ricci-flat.
  \item[\rm{(b)} :] Ricci-flat, $S_L^2=O \ \mathrm{and} \ \mathrm{rank}\, S_{L} \leq 1$. 
  \item[\rm{(c)} :] Ricci-flat (when $d=0$), Ricci-flat, $S'=S'^2\ \mathrm{and} \ [S',T']=O$ (when $d=1$).
\end{enumerate}
\end{prop}

\begin{proof}
In case (a), under the Ricci-flatness, 
we have the operators 
\[ L_{X_{i}} = 0 \ (1 \leq i \leq n-1), \quad 
L_{X_{n}} = 
\begin{bmatrix}
\, T & \bm{0} \, \\
\, \bm{0}^{T} & 0 \, \\
\end{bmatrix}
. \]
A direct calculation from these equations yields $R = 0$, which shows that Ricci-flatness implies flatness.

In case (b), under the Ricci-flatness, 
we have the operators 
\begin{align}
L_{X_{i}} = 
\begin{bmatrix}
\, O & -\bm{s}_{i} \, \\
\, \bm{s}_{i}^{T}J & 0 \, \\
\end{bmatrix}
\ (1 \leq i \leq n-1), \quad 
L_{X_{n}} = 0, \label{eq:lxi_b}
\end{align}
where we set 
$S_{L} = [\bm{s}_{1}, \ldots, \bm{s}_{n-1}] 
\in \mathrm{M}_{n-1}\mathbb{R}$. 
From them, we have 
\[ R(X_{i}, X_{j}) = 
\begin{bmatrix}
\, -(\bm{s}_{i}\bm{s}_{j}^{T} - \bm{s}_{j}\bm{s}_{i}^{T})J & \bm{0} \, \\
\, \bm{0}^{T} & 0 \, \\
\end{bmatrix}
\ (1 \leq i, j \leq n-1), \quad 
R(X_{n}, X_{i}) = 
\begin{bmatrix}
\, O & -S_{L}\bm{s}_{i} \, \\
\, \bm{s}_{i}^{T}S_{L}^{T}J & 0 \, \\
\end{bmatrix}
. \]
Therefore, the metric is flat if and only if 
\[ \bm{s}_{i}\bm{s}_{j}^{T} = \bm{s}_{j}\bm{s}_{i}^{T} \ (1 \leq i, j \leq n-1), \quad S_{L}^{2} = O. \]
The first condition is equivalent to $\mathrm{rank}\, S_{L} \leq 1$. 
Thus, we 
derive the desired flat conditions. 

In case (c), under the Ricci-flatness, we have the operators 
\begin{align}
L_{X_{i}} = 
\begin{bmatrix}
\, O & \bm{0} & -\bm{s}_{i}' \, \\
\, \bm{s}_{i}'^{T} & 0 & 0 \, \\
\, \bm{0}^{T} & 0 & 0 \, \\ 
\end{bmatrix}
\ (1 \leq i \leq n-2), \quad L_{X_{n-1}} = 0, \quad 
L_{X_{n}} = 
\begin{bmatrix}
\, T' & \bm{0} & -\bm{c} \, \\
\, \bm{c}^{T} & d & 0 \, \\
\, \bm{0}^{T} & 0 & -d \, \\
\end{bmatrix}
, \label{eq:lxi_c}
\end{align}
where we set 
$S' = [\bm{s}_{1}', \ldots, \bm{s}_{n-2}'] \in \mathrm{M}_{n-2}\mathbb{R}$ 
and $d = 0, 1$. 
From them, we find that
\begin{align}
R(X_{i}, X_{j}) = R(X_{n-1}, X_{i}) = 0 \ (1 \leq i, j \leq n-2), \quad 
R(X_{n}, X_{n-1}) = 0. \label{eq:rxi} 
\end{align}
Also, $R(X_{n}, X_{i}) = 0$ for any $i = 1, \ldots, n-2$ 
if and only if $S'\,\!^{2} - dS' + [S', T'] = O$.
By diagonalizing $S'$, it is equivalent to 
$S'\,\!^{2} = dS', \, [S', T'] = O$.
When $d = 0$, 
we have $S'\,\!^{2} = O$, that is, $S' = O$. 
When $d = 1$, 
we have the flat conditions $S' = S'\,\!^{2}, \, [S', T'] = O$. 
\end{proof}

Combining the flat conditions and Ricci-flat conditions, 
we have the following proposition. 

\begin{prop} \label{Ric_flat2} 
The metric \eqref{emet} can be Ricci-flat and non-flat only in the following three cases:
\begin{align}
    \langle \, , \, \rangle &=\langle \, , \, \rangle_{\mathrm{b}},\quad A=\begin{bmatrix}
\, \, \alpha & -\beta \, \\
\, \, \beta & \alpha \, \\
\end{bmatrix}
\oplus \, 
\mathrm{Diag}(\lambda_{3}, \ldots, \lambda_{n-1}),
\label{eq:bi}\\
\langle \, , \, \rangle&=\langle \, , \, \rangle_{\mathrm{b}},\quad 
A=\begin{bmatrix}
\, -\delta & -\delta & -1 \, \\
\, \delta & \delta & 1 \, \\
\, 1 & 1 & 0 \, \\
\end{bmatrix}
\oplus O_{n-4},
\label{eq:biii}\\
\langle \, , \, \rangle&=\langle \, , \, \rangle_{\mathrm{c}},\quad 
A = 
\begin{bmatrix}
\, A' & \bm{0} \, \\
\, \bm{c}^{T} & 1 \, \\
\end{bmatrix},
\label{eq:c1}
\end{align}
where $\alpha\le 0,\ \beta>0,\ \lambda_3,\dots,\lambda_{n-1},\delta\in\mathbb{R},\ A' \in \mathrm{M}_{n-2}\mathbb{R}, \, \bm{c} \in \mathbb{R}^{n-2}$.
In the other cases, the Ricci-flat metrics must be flat.
\end{prop} 

\begin{proof}
In case (a), a Ricci-flat metric is always flat, as shown in Proposition~\ref{flat}.\\
In case (b), we check the Ricci-flat conditions 
for each matrix in Proposition~\ref{normal_form}. 
When $S_{L}$ is a diagonal matrix, 
the Ricci-flat condition $\mathrm{tr}\, S_L^2=0$ implies $S_{L} = O$. 
This trivially satisfies the flatness conditions $S_L^2=O$ and $\mathrm{rank}\, S_L\le 1$ given in Proposition~\ref{flat}.  
When $S_{L}$ is in case (i), 
$\beta\neq0$ gives $\text{rank}\, S_L\ge2$, 
hence the metric cannot be flat.
When $S_{L}$ is in case (ii), that is, 
\begin{equation*}
S_{L} = 
\begin{bmatrix}
\, -1 + \gamma & -1 \, \\
\, 1 & 1 + \gamma \, \\
\end{bmatrix}
\oplus \, 
\mathrm{Diag}(\lambda_{3}, \ldots, \lambda_{n-1}), 
\end{equation*}
the Ricci-flat conditions imply $\gamma = \lambda_{3} = \cdots = \lambda_{n-1} = 0$. It satisfies the flat condition.
When $S_{L}$ is in case (iii), that is, 
\begin{equation*}
S_{L} = 
\begin{bmatrix}
\, -\delta & -\delta & -1 \, \\
\, \delta & \delta & 1 \, \\
\, 1 & 1 & 0 \, \\
\end{bmatrix}
\oplus \, 
\mathrm{Diag}(\lambda_{4}, \ldots, \lambda_{n-1}), 
\end{equation*}
the Ricci-flat condition $\mathrm{tr}\, S_L^2=0$ implies $\lambda_{4} = \cdots = \lambda_{n-1} = 0$. 
Since $\text{rank}\, S_L\ge 2$, the metric cannot be flat.

In case (c), the Ricci-flat metric is always flat when $d=0$ as shown in Proposition~\ref{flat}.
When $d=1$, the Ricci-flat conditions do not coincide with the flat conditions in general.
\end{proof}

We have obtained all the possible Ricci-flat but non-flat cases \eqref{eq:bi}, \eqref{eq:biii} and \eqref{eq:c1}.
The first one is the most important since it corresponds to the Petrov solution in the four-dimensional case.

\begin{cor} \label{cor:met}
    For the case \eqref{eq:bi}, the left-invariant Lorentzian metric is \eqref{main} and the Ricci-flat condition is \eqref{Ric_flat_cond}.
\end{cor}
\begin{proof}
    By substituting $A$ of \eqref{eq:bi} into \eqref{emet}, we have the metric \eqref{main}.
    By substituting $A$ into the Ricci-flat conditions \eqref{eq:trsl}, we have \eqref{Ric_flat_cond}.
\end{proof}

\section{Isometries of Ricci-flat solutions} \label{Petrov}

In this section, we complete the proof of the main theorem~\ref{th}.
The metric \eqref{main} and the Ricci-flat conditions \eqref{Ric_flat_cond} in the main theorem are obtained in Corollary~\ref{cor:met}.
The last condition \eqref{simp} in the main theorem is obtained in this section.

In subsection~\ref{ssec:pw}, we show that the metrics in cases \eqref{eq:biii} and \eqref{eq:c1} are plane waves, which implies that their isometry groups are multiply transitive.
All non-flat Ricci-flat homogeneous solutions with a multiply transitive isometry group are plane waves in the four-dimensional case~\cite[Theorem 12.1]{MR2003646}.
In subsection~\ref{ssec:isotropy}, we show that the identity component of the full isometry group of case \eqref{eq:bi} can be simply transitive under the condition \eqref{simp}.

The action of the isometry group is called \textit{multiply transitive}, \textit{almost simply transitive}, or \textit{simply transitive} if it is transitive and the isotropy group is positive-dimensional, discrete, or trivial, respectively. 

\subsection{Multiply-transitive cases} \label{ssec:pw}

It is known that the action of the isometry group of a homogeneous plane wave is multiply-transitive. 
See \cite[Proposition~4.5]{MR3539491} for details. 

An $n$-dimensional Lorentzian manifold 
$(M, g_{M})$ is called a \textit{pp-wave} if 
there exists a parallel null vector field $V$ such that 
it is transversally flat, that is, 
for any $X, Y \in \Gamma(V^{\bot})$ it holds that  $R(X,Y) = 0$, 
where $R$ denotes the curvature tensor of $(M, g_{M})$ and 
\[ V^{\bot} := \{ X \in \Gamma(TM) \mid g_{M}(X, V) = 0 \}, \]
which is a subbundle of the tangent bundle of $M$ with rank $n-1$.
We should remark that $V^{\bot}$ is not the complementary distribution of 
the line bundle generated by $V$. 
Actually, it trivially holds that $V \in \Gamma(V^{\bot})$. 
A Lorentzian manifold $(M, g_{M})$ is called a {\it plane wave} 
if it is a pp-wave and it holds that $\nabla_{X}R = 0$
for any $X \in \Gamma(V^{\bot})$.

A Lorentzian manifold $(M, g_{M})$ is called a \textit{pr-wave} if 
there exists a recurrent null vector field $V$ such that 
it is transversally flat. 
A vector field $V$ is \textit{recurrent} 
if there exists a $1$-form $\omega$ of $M$ such that 
$\nabla V = V \otimes \omega$, 
where $\nabla$ denotes the Levi-Civita connection of $(M, g_{M})$. 
Obviously, if it is a pp-wave, then it is a pr-wave. 
The converse is not true in general, but the following is known. 

\begin{prop}[{\cite[Proposition~6.11]{MR2254049}}] \label{pr-wave}
Let $(M, g_{M})$ be an $n$-dimensional pr-wave. 
It is a pp-wave if and only if 
its Ricci operator is totally isotropic, 
that is, for any vector fields $X, Y$ it holds that 
$g_{M}(Q(X),Q(Y))=0$,
where $Q$ is the Ricci operator of $(M, g_{M})$. 
In particular, a Ricci-flat pr-wave is a pp-wave. 
\end{prop} 

\begin{lem} \label{conn_curv_comm}
For the Lorentzian almost abelian Lie algebra $(\mathfrak{g}, \langle \, , \, \rangle)$ 
where $\langle \, , \, \rangle$ is one of the three cases in Lemma~\ref{metric}, it holds that 
\[ L_{X_{i}}R(X_{j}, X_{a}) = R(X_{j}, X_{a}) L_{X_{i}}, \]
where $1 \leq i, j \leq n$ and $1 \leq a \leq n$.
\end{lem} 

\begin{proof}
When $1 \leq a \leq n-1$, we calculate that 
\begin{align*}
L_{X_{i}}R(X_{j}, X_{a}) = L_{X_{i}}L_{X_{j}}L_{X_{a}} - L_{X_{i}}L_{X_{a}}L_{X_{j}} = L_{X_{j}}L_{X_{a}}L_{X_{i}} - L_{X_{a}}L_{X_{j}}L_{X_{i}} = R(X_{j}, X_{a}) L_{X_{i}}
\end{align*}
since $[X_{i}, X_{j}] = [X_{i}, X_{a}] = [X_{j}, X_{a}] = 0$. 
When $a = n$, 
a similar computation shows that
\begin{align*}
L_{X_{i}}R(X_{j}, X_{n}) &= L_{X_{i}}L_{X_{j}}L_{X_{n}} - L_{X_{i}}L_{X_{n}}L_{X_{j}} + \sum_{k=1}^{n-1}A_{j}^{k}L_{X_{i}}L_{X_{k}} \\
&= L_{X_{j}}\left(L_{X_{n}}L_{X_{i}} + \sum_{l=1}^{n-1}A_{i}^{l}L_{X_{l}}\right) - \left(L_{X_{n}}L_{X_{i}} + \sum_{l=1}^{n-1}A_{i}^{l}L_{X_{l}}\right)L_{X_{j}} + \sum_{k=1}^{n-1}A_{j}^{k}L_{X_{i}}L_{X_{k}} \\
&= L_{X_{j}}L_{X_{n}}L_{X_{i}} - L_{X_{n}}L_{X_{j}}L_{X_{i}} + \sum_{k=1}^{n-1}A_{j}^{k}L_{X_{k}}L_{X_{i}} = R(X_{j}, X_{n}) L_{X_{i}}, 
\end{align*}
where we should remark that 
\[ R(X_{j}, X_{n}) = L_{X_{j}}L_{X_{n}} - L_{X_{n}}L_{X_{j}} + \sum_{k=1}^{n-1}A_{j}^{k}L_{X_{k}}. \]
This completes the proof. 
\end{proof}

Henceforth, 
we determine Ricci-flat plane waves for almost abelian Lie groups $(G, ds^{2})$ with left-invariant Lorentzian metrics. 

\begin{lem} \label{b-iii}
For the Ricci-flat Lorentzian almost abelian Lie algebra 
$(\mathfrak{g}, \langle \, , \, \rangle_{\mathrm{b}})$, 
setting the associated matrix $A$ as \eqref{eq:biii}, the corresponding Lorentzian manifold $(G, ds^{2})$ is a plane wave. 
\end{lem} 

\begin{proof}
First, we show that it is a pp-wave. 
We define $V = X_{1} - X_{2} \in \mathfrak{g}$, 
and it gives a parallel null vector field. 
In fact, we have $L_{X_{a}}V = 0$ for $a = 1, \ldots, n$. 
Moreover, at the identity element, the fiber of the distribution $V^{\bot}$ 
is spanned by $V, X_{3}, \ldots, X_{n}$. 
Then for any $3 \leq i, j \leq n$ we have 
\[ R(V, X_{i}) = R(X_{i}, X_{j}) = 0. \]
Thus, it is a pp-wave. 
Next, we show that it becomes a plane wave. 
When $4 \leq i \leq n$, $\nabla_{X_{i}}R = 0$ is trivial. 
Regarding the definition of $\nabla_{X}R$, see Appendix~\ref{loc_symm_sp}. Let $1 \leq i, j \leq n-1$. From Lemma~\ref{conn_curv_comm}, we have
\begin{align*}
(\nabla_{X_{3}}R)(X_{i}, X_{j}) &= L_{X_{3}}R(X_{i}, X_{j}) - R(L_{X_{3}}X_{i}, X_{j}) - R(X_{i},L_{X_{3}}X_{j}) - R(X_{i}, X_{j})L_{X_{3}} = 0, \\
(\nabla_{X_{3}}R)(X_{i}, X_{n}) &= L_{X_{3}}R(X_{i}, X_{n}) - R(L_{X_{3}}X_{i}, X_{n}) - R(X_{i},L_{X_{3}}X_{n}) - R(X_{i}, X_{n})L_{X_{3}} \\ 
&= R(V, X_{i}) = 0
\end{align*}
since $L_{X_{3}}X_{i}$ is parallel to $X_{n}$ or zero, and $L_{X_{3}}X_{n} = X_{1}-X_{2} = V$ are satisfied from \eqref{eq:lxi_b}. 
Similarly, we have 
\begin{align*}
(\nabla_{X_{1}}R)(X_{i}, X_{j}) &= -R(L_{X_{1}}X_{i}, X_{j}) - R(X_{i}, L_{X_{1}}X_{j}) = -R(L_{X_{2}}X_{i}, X_{j}) - R(X_{i}, L_{X_{2}}X_{j}) \\
&= (\nabla_{X_{2}}R)(X_{i}, X_{j}) \\
(\nabla_{X_{1}}R)(X_{i}, X_{n}) &= -R(L_{X_{1}}X_{i}, X_{n}) - R(X_{i}, L_{X_{1}}X_{n}) = -R(L_{X_{2}}X_{i}, X_{n}) - R(X_{i}, L_{X_{2}}X_{n}) \\
&= (\nabla_{X_{2}}R)(X_{i}, X_{n}). 
\end{align*}
Hence, for any $X \in \mathrm{span}_{\mathbb{R}}\{V, X_{3}, \ldots, X_{n} \}$ we have $\nabla_{X}R = 0$, that is, it is a plane wave.
\end{proof}

\begin{lem} \label{c-1}
For the Ricci-flat Lorentzian almost abelian Lie algebra 
$(\mathfrak{g}, \langle \, , \, \rangle_{\mathrm{c}})$, 
setting the associated matrix $A$ as \eqref{eq:c1}, the corresponding Lorentzian manifold $(G, ds^{2})$ is a plane wave. 
\end{lem}

\begin{proof} 
First, we show that it is a pr-wave. 
We define $V = X_{n-1} \in \mathfrak{g}$, 
and it is not parallel, but a recurrent null vector. 
In fact, we have $L_{X_{i}}V = 0$ for 
$i = 1, \ldots, n-1$ and $L_{X_{n}}V = V$. 
Moreover, at the identity element, the fiber of the distribution $V^{\bot}$ 
is spanned by $X_{1}, \ldots, X_{n-1}$. 
Then, for any $1 \leq i, j \leq n-1$, we have $R(X_{i}, X_{j}) = 0$. 
Thus, it is a pr-wave. 
By Proposition~\ref{pr-wave}, it must be a pp-wave. 
Next, we show that it becomes a plane wave. 
Let $1 \leq i, j, k \leq n-2$. From Lemma~\ref{conn_curv_comm}, we have 
\begin{align*}
(\nabla_{X_{i}}R)(X_{j}, X_{k}) &= -R(L_{X_{i}}X_{j}, X_{k}) -R(X_{j}, L_{X_{i}}X_{k}) = 0, \\
(\nabla_{X_{i}}R)(X_{j}, X_{n-1}) &= -R(L_{X_{i}}X_{j}, X_{n-1}) -R(X_{j}, L_{X_{i}}X_{n-1}) = 0, \\
(\nabla_{X_{i}}R)(X_{j}, X_{n}) &= -R(L_{X_{i}}X_{j}, X_{n}) -R(X_{j}, L_{X_{i}}X_{n}) = 0, \\
(\nabla_{X_{i}}R)(X_{n-1}, X_{n}) &= -R(L_{X_{i}}X_{n-1}, X_{n}) -R(X_{n-1}, L_{X_{i}}X_{n}) = 0 
\end{align*}
since we see $L_{X_{i}}X_{j}$ is parallel to $X_{n-1}$ and $L_{X_{i}}X_{n-1} = 0$, $L_{X_{i}}X_{n} \in \mathrm{span}_{\mathbb{R}}\{X_{1}, \ldots, X_{n-2} \}$ from \eqref{eq:lxi_c} and \eqref{eq:rxi}.
Moreover, we see $\nabla_{X_{n-1}}R = 0$ by $L_{X_{n-1}} = 0$. 
Hence, for any $X \in \mathrm{span}_{\mathbb{R}}\{X_{1}, \ldots, X_{n-1} \}$ we have $\nabla_{X}R = 0$, that is, it is a plane wave.
\end{proof}

\subsection{Simply-transitive case} \label{ssec:isotropy}

Here, we consider the metric (\ref{main}) that satisfies the Ricci-flat conditions \eqref{Ric_flat_cond}.
Since $\alpha$ and  $\beta$ are uniquely determined by $\lambda_{3}, \ldots, \lambda_{n-1}$, we denote the solution by 
\[ 
\mathbb{P}(\lambda_{3}, \ldots, \lambda_{n-1}) 
= (\mathbb{R}^{n}, ds^{2}).
\]
Note that $G$ is diffeomorphic to $\mathbb{R}^{n}$ and 
$\alpha \leq 0$ if and only if $\lambda_{3} + \cdots + \lambda_{n-1} \geq 0$. 

\begin{prop} \label{trans_isom}
Let $n \geq 5$. 
For the solution $\mathbb{P}(\lambda_{3}, \ldots, \lambda_{n-1})$, 
the full isometry group acts on it almost simply transitively 
if and only if $\lambda_{3} > \ldots > \lambda_{n-1}$ holds. 
\end{prop}

\begin{proof}
By reordering $X_3,\dots,X_{n-1}$, we can always choose $\lambda_3\ge\dots\ge\lambda_{n-1}$.
Assume that there exists $3 \leq i \leq n-2$ such that $\lambda_{i} = \lambda_{i+1}$. 
Since the metric~(\ref{main}) contains the term $e^{-2\lambda_{i}x_{n}}(dx_{i}^{2} + dx_{i+1}^{2})$, the isotropy group contains the orthogonal group $\mathrm{O}(2)$. 
Thus, the condition $\lambda_{3} > \ldots > \lambda_{n-1}$ is necessary for the isometry group to act simply transitively.

Conversely, let $\lambda_{3} > \ldots > \lambda_{n-1}$. 
Here, we define two symmetric groups as follows: 
\begin{align*}
\mathrm{O}(\mathfrak{g}, \langle \, , \, \rangle_{\mathrm{b}}) 
&:= \{ \Phi \in \mathrm{GL}(\mathfrak{g}) = \mathrm{GL}_{n}\mathbb{R} \mid 
\langle \Phi(\cdot), \Phi(\cdot) \rangle_{\mathrm{b}} = \langle \, , \, \rangle_{\mathrm{b}} \} 
= \mathrm{O}(1, n-1), \\
\mathrm{Sym}(\mathfrak{g}, \langle \, , \, \rangle_{\mathrm{b}}) 
&:= \{ \Phi \in \mathrm{O}(\mathfrak{g}, \langle \, , \, \rangle_{\mathrm{b}}) \mid 
R(\Phi(X), \Phi(Y)) = \Phi R(X, Y) \Phi^{-1} \ 
(\forall X, Y \in \mathfrak{g}) \},
\end{align*}
where we recall that $R$ denotes the curvature operator. 
By denoting the isotropy group of the isometry group of 
$\mathbb{P}(\lambda_{3}, \ldots, \lambda_{n-1})$ at the identity element as $\mathrm{Isom}(\mathbb{P}(\lambda_{3}, \ldots, \lambda_{n-1}))_{\bm{0}}$, we observe that 
\[ \mathrm{Aut}(\mathfrak{g}, \langle \, , \, \rangle_{\mathrm{b}}) \subset 
\mathrm{Isom}(\mathbb{P}(\lambda_{3}, \ldots, \lambda_{n-1}))_{\bm{0}} \subset 
\mathrm{Sym}(\mathfrak{g}, \langle \, , \, \rangle_{\mathrm{b}}) \]
in accordance with \cite[Remark~4.3]{MR4528847}. 
We define two subsets of $\mathrm{O}(1, n-1)$ as 
\begin{align*}
\Gamma_{1} &:= \{ \mathrm{Diag}
(\tau, \tau, \tau_{3}, \ldots, \tau_{n-1}, \tau_{n}) 
\mid \tau = \pm 1, \ \tau_{l} = \pm 1 \ (3 \leq l \leq n) \}, \\ 
\Gamma_{2} &:= \left\{ 
\begin{bmatrix}
    \, \, \tau' & & & & \\
     & -\tau' & & & & \\
     & & & & \tau'_{3} \\
     & & & \iddots & \\
     & & \tau'_{n-1} & & \\
     & & & & & \tau'_{n} \, \,  \\
\end{bmatrix} 
\, \middle| \ 
\tau' = \pm 1, \ \tau'_{l} = \pm 1 \ (3 \leq l \leq n) \right\}, 
\end{align*}
and we compute the conditions $R(\Phi(X_{a}), \Phi(X_{b})) = \Phi R(X_{a}, X_{b}) \Phi^{-1}$ 
for $\{ X_{1}, \ldots, X_{n} \}$. 
In special cases where there exists a positive integer $h \geq 2$ and 
$\kappa_{1} > \cdots > \kappa_{h-1} > 0$ 
such that 
\begin{align*}
(\lambda_{3}, \ldots, \lambda_{h+1}, \lambda_{h+2}, \ldots, \lambda_{2h}) 
&= (\kappa_{1}, \ldots, \kappa_{h-1}, 
-\kappa_{h-1}, \ldots, -\kappa_{1}) \quad (n = 2h + 1 \geq 5), \\
(\lambda_{3}, \ldots, \lambda_{h+1}, \lambda_{h+2}, \lambda_{h+3}, \ldots, \lambda_{2h+1}) 
&= (\kappa_{1}, \ldots, \kappa_{h-1}, 0, 
-\kappa_{h-1}, \ldots, -\kappa_{1}) \quad (n = 2h + 2 \geq 6), 
\end{align*}
we have 
\[ 
\mathrm{Sym}(\mathfrak{g}, \langle \, , \, \rangle_{\mathrm{b}}) 
= \Gamma_{1} \cup \Gamma_{2}, 
\] 
otherwise, we have 
\[ 
\mathrm{Sym}(\mathfrak{g}, \langle \, , \, \rangle_{\mathrm{b}}) 
= \Gamma_{1}. 
\]
Since we see that the isotropy subgroup is finite, 
the action of the full isometry group is almost simply transitive. 
\end{proof} 

We call the Ricci-flat solution $\mathbb{P}(\lambda_{3}, \ldots, \lambda_{n-1})$ satisfying the simply-transitive condition~\eqref{simp} the generalized Petrov solution.

\begin{cor} \label{pl3}
    In the case $n=4$, the generalized Petrov solution $\mathbb{P}(\lambda_{3})$ is isometric to the Petrov solution \eqref{Petrov_sol}. Moreover, in this case, the full isometry group acts on it almost simply transitively. 
\end{cor}
\begin{proof}
    The simply-transitive condition on the generalized Petrov solution~\eqref{simp} is trivially satisfied. 
    The Ricci-flat conditions~\eqref{Ric_flat_cond} give $\beta=-\sqrt{3}\alpha,\ \lambda_3=-2\alpha$. 
    Rescaling the coordinates as
    \[x_a'=-2\alpha x_a\ (a=1,\dots,4) \] 
    and defining $k^2=4\alpha^2$ 
    yields \eqref{Petrov_sol}. 
    Moreover, in this case, 
    the same computation as in Proposition~\ref{trans_isom} yields
    \[ \mathrm{Sym}(\mathfrak{g}, \langle \, , \, \rangle_{\mathrm{b}}) 
= \Gamma_{1} \]
    Consequently, the full isometry group acts on it almost simply transitively. 
\end{proof}

\subsubsection*{Proof of Theorem~\ref{th}}
Combining Propositions~\ref{Ric_flat2},  \ref{trans_isom}, Corollary~\ref{cor:met} and Lemmas~\ref{b-iii}, \ref{c-1}, we obtain the claim of Theorem~\ref{th}. 
By Proposition~\ref{trans_isom}, the isotropy subgroup of the full isometry group is finite. 
Hence the isotropy subgroup of the identity component of the full isometry group is also finite. 
Since the identity component has dimension \(n\), all its orbits are open in the underlying manifold \(\mathbb{R}^{n}\). 
As \(\mathbb{R}^{n}\) is connected, the action is transitive. 
Moreover, the orbit map from the identity component to \(\mathbb{R}^{n}\) is then a covering map with finite fiber. 
Since the identity component is connected and \(\mathbb{R}^{n}\) is simply connected, this covering must be trivial. 
Thus the isotropy subgroup of the identity component is trivial. 
Therefore, the identity component of the full isometry group acts simply transitively. \qed


\begin{rem} 
Note that we have \cite[Corollary~2.8]{MR2864471}
\[ \mathrm{Aut}(\mathfrak{g}, \langle \, , \, \rangle_{\mathrm{b}}) 
= \mathrm{Isom}(\mathbb{P}(\lambda_{3}, \ldots, \lambda_{n-1}))_{\bm{0}}. \]
For $\Phi \in \Gamma_{1}$, it is an automorphism if and only if $\tau_n = 1$.
For $\Phi \in \Gamma_{2}$, 
it is an automorphism if and only if $\tau_n' = -1$. 
Hence 
$\mathrm{Isom}(\mathbb{P}(\lambda_{3}, \ldots, \lambda_{n-1}))_{\bm{0}} \varsubsetneq 
\mathrm{Sym}(\mathfrak{g}, \langle \, , \, \rangle_{\mathrm{b}})$. 

Moreover, we should note that the solution that admits a simply-transitive isometric action may be decomposable as a Lie group. 
See Appendix~\ref{indecom} for the decomposability of the solutions. 
\end{rem} 

\begin{cor}
The isotropy subgroup
$\mathrm{Isom}(\mathbb{P}(\lambda_{3}, \ldots, \lambda_{n-1}))_{\bm{0}}$ is isomorphic to the following finite group: 
\begin{align*}
\mathbb{Z}_{2}^{n-2} \quad 
&(\mathrm{when} \ n \geq 4 \ \mathrm{and} \ 
\mathrm{Sym}(\mathfrak{g}, \langle \, , \, \rangle_{\mathrm{b}})  = \Gamma_{1}), \\ 
\mathbb{Z}_{2}^{n-4} \rtimes D_{4} \quad 
&(\mathrm{when} \ n \geq 5 \ \mathrm{and} \ \mathrm{Sym}(\mathfrak{g}, \langle \, , \, \rangle_{\mathrm{b}}) = \Gamma_{1} \cup \Gamma_{2}),
\end{align*}
where $D_{4}$ denotes the dihedral group of order $8$. 
Moreover, when $n = 5$ only, the semidirect product $\mathbb{Z}_{2} \rtimes D_{4}$ becomes the product $\mathbb{Z}_{2} \times D_{4}$. 
\end{cor}

\begin{proof} 
First of all, we define two subsets as follows; 
\[ \Gamma_{1}' := \{ \Phi \in \Gamma_{1} \mid \tau_{n} = 1 \}, \quad 
\Gamma_{2}' := \{ \Phi \in \Gamma_{2} \mid \tau'_{n} = -1 \}. \]
In the former case, it is easy to see that 
\begin{align*} 
\mathrm{Isom}(\mathbb{P}(\lambda_{3}, \ldots, \lambda_{n-1}))_{\bm{0}} 
= \mathrm{Aut}(\mathfrak{g}, \langle \, , \, \rangle_{\mathrm{b}}) 
= \Gamma_{1}', 
\end{align*}
which is obviously isomorphic to $\mathbb{Z}_{2}^{n-2}$. 
In the latter case, we similarly see that 
\begin{align*} 
\mathrm{Isom}(\mathbb{P}(\lambda_{3}, \ldots, \lambda_{n-1}))_{\bm{0}} 
= \mathrm{Aut}(\mathfrak{g}, \langle \, , \, \rangle_{\mathrm{b}}) 
= \Gamma_{1}' \cup \Gamma_{2}'. 
\end{align*}
In addition, we give a generator of the group 
$\Gamma' := \Gamma_{1}' \cup \Gamma_{2}'$ as follows; 
\begin{align*}
a_{1} &= \mathrm{Diag}(-1, -1, 1, \ldots, 1) \in \Gamma_{1}', \\  
a_{2} &= \mathrm{Diag}(1, 1, 1, -1, 1, \ldots, 1, 1) \in \Gamma_{1}', \\
a_{3} &= \mathrm{Diag}(1, 1, 1, 1, -1, 1, \ldots, 1, 1) \in \Gamma_{1}', \\
\vdots \\
a_{n-4} &= \mathrm{Diag}(1, 1, 1, \ldots, 1, -1, 1, 1) \in \Gamma_{1}', \\ 
s &= \mathrm{Diag}(-1, -1, -1, 1, \ldots, 1) \in \Gamma_{1}', \\ 
r &= 
\begin{bmatrix}
     1 & & & & & & \\
     & -1 & & & & & & \\
     & & & & & 1 & \\
     & & & & \iddots & & \\
     & & & 1 & & & \\
     & & -1 & & &\\
     & & & & & &-1 \\
\end{bmatrix}
\in \Gamma_{2}', 
\end{align*}
then we see that $\langle r, s, a_{1}, \ldots, a_{n-4} \rangle = \Gamma'$, 
and $H := \langle r, s \rangle \cong D_{4}, \ N := \langle a_{1}, \ldots, a_{n-4} \rangle \cong \mathbb{Z}_{2}^{n-4}$ by direct computation. 
It is obvious that 
$N$ is a normal subgroup of $\Gamma'$ and $N \cap H = \{ I_{n} \}$. 
For the conjugations $\varphi_{s}, \, \varphi_{r}$ for $s, r \in H$, 
we see that 
\[
\varphi_{s} = 1_{N}, \quad \varphi_{r}(a_{1}) = ra_{1}r^{-1} = a_{1}, \quad 
\varphi_{r}(a_{l}) = ra_{l}r^{-1} = a_{n-2-l} \quad 
(n \geq 6 \ \mathrm{and} \ 2 \leq l \leq n-4).  
\]
In particular, we have $\Gamma' = N H$ from the above results, 
then we obtain the claim. 
\end{proof}

\begin{rem}
Since the isotropy subgroup of the isometry group of the generalized Petrov solution is discrete,
the isometry group is given by
\[
\mathrm{Isom}(\mathbb{P}(\lambda_{3}, \ldots, \lambda_{n-1}))
= L(G) \rtimes \mathrm{Aut}(\mathfrak{g}, \langle \, , \, \rangle_{\mathrm{b}}),
\]
where $L(G)$ denotes the group of left translations of $G$;
in particular, the identity component is precisely $L(G)$, which is naturally isomorphic to $G$. 
For details, see \cite[Lemma~2.2]{MR2864471}.
Consequently, non-isomorphic almost abelian Lie groups cannot be isometric.
When $n = 4$, the isometry class is unique up to scaling of the metric
and coincides with the classical Petrov solution.
In contrast, for $n \ge 5$, there exist infinitely many isometry classes up to scaling, since there are infinitely many isomorphism classes satisfying the Ricci-flat conditions \eqref{Ric_flat_cond}. 
\end{rem}

\section{Properties of the generalized Petrov solution} \label{property}
In this section, we show that the generalized Petrov solution $\mathbb{P}(\lambda_{3}, \ldots, \lambda_{n-1})$ is geodesically complete and admits closed timelike curves.

For Propositions~\ref{geod_cplt} and \ref{causality}, 
we should note that the Ricci-flat conditions \eqref{Ric_flat_cond} and the conditions for simply-transitive action \eqref{simp} are not needed. 
Actually, for arbitrary parameters $\alpha, \beta, \lambda_{3}, \ldots, \lambda_{n-1}$ and $n \geq 4$, these claims hold as long as $\beta$ is nonzero.
We denote the metric by $g_{\mathbb{R}^{n}}$ instead of $ds^{2}$ in this section. 

\subsection{Completeness}
A homogeneous Riemannian manifold is complete~\cite[Theorem~7.19]{MR2371700}.
This is not true for homogeneous Lorentzian manifolds.
However, we can show the Proposition below.

\begin{prop} \label{geod_cplt}
Let $\lambda_{3}, \ldots, \lambda_{n-1}\in \mathbb{R}$, $\alpha\le0,\ \beta > 0$, and $g_{\mathbb{R}^{n}}$ be the Lorentzian metric on $\mathbb{R}^{n}$ given by (\ref{main}). 
Then the homogeneous Lorentzian manifold 
$(\mathbb{R}^{n}, g_{\mathbb{R}^{n}})$ is geodesically complete. 
\end{prop} 

\begin{proof}
To investigate the geodesic completeness of a Lie group with a left-invariant metric, it suffices to show the existence of entire solutions of the corresponding Arnold--Euler equation. 
According to \cite[2.2~General fact]{MR1430782}, 
when the restriction of $\langle \, , \, \rangle$ on $\mathfrak{a}$ is non-degenerate for a Lorentzian almost abelian Lie algebra $(\mathfrak{g}, \langle \, , \, \rangle)$, the corresponding equations are given by 
\begin{equation}
 \dot{u} = u_{n} \left[ (\mathrm{ad} \, X_{n})^{\ast}u \right], 
\quad \dot{u}_{n} = - \frac{\langle [X_{n}, u], u \rangle}{\langle X_{n}, X_{n} \rangle}, \label{AE_eq}
\end{equation}
where $u = u(t) \in \mathfrak{a}$ and 
$(\mathrm{ad} \, X_{n})^{\ast}$ denotes the adjoint of $\mathrm{ad} \, X_{n}$ with respect to $\langle \,, \, \rangle$. 
In the case that the structure constants are expressed by 
\[ A = 
\begin{bmatrix}
\, \, \alpha & -\beta \, \\
\, \, \beta & \alpha \, \\
\end{bmatrix}
\oplus \, 
\mathrm{Diag}(\lambda_{3}, \ldots, \lambda_{n-1}), \]
and the metric is given by $\langle \, , \, \rangle_{\mathrm{b}}$, 
the equation (\ref{AE_eq}) is reduced to 
\begin{equation}
\left\{ \,
    \begin{aligned}
     \dot{u}_{1} &= u_{n}(\alpha u_{1} - \beta u_{2}), \\
     \dot{u}_{2} &= u_{n}(\beta u_{1} + \alpha u_{2}), \\
     \dot{u}_{3} &= \lambda_{3}u_{n}u_{3}, \\
     &\vdots \\
     \dot{u}_{n-1} &= \lambda_{n-1}u_{n}u_{n-1}, \\
     \dot{u}_{n} &= \alpha(u_{1}^{2} - u_{2}^{2}) - 2\beta u_{1}u_{2} - \lambda_{3}u_{3}^{2} - \cdots - \lambda_{n-1}u_{n-1}^{2}. \\
    \end{aligned}
\right. \label{ODE_1}
\end{equation}
The equation has a conserved quantity 
\begin{align*}
-u_{1}^{2} + u_{2}^{2} + \cdots + u_{n}^{2} = C,
\end{align*}
where $C$ is a constant. 
When the initial values satisfy $(u_{1}(t_0), u_{2}(t_0)) = (0, 0)$, 
the system given by 
$(u_1, u_2, u_3, \ldots, u_{n}) = (0, 0, v_{3}, \ldots, v_{n})$ satisfies
\begin{equation}
\left\{ \,
    \begin{aligned}
     \dot{v}_{3} &= \lambda_{3}v_{n}v_{3}, \\
     &\vdots \\
     \dot{v}_{n-1} &= \lambda_{n-1}v_{n}v_{n-1}, \\
     \dot{v}_{n} &= - \lambda_{3}v_{3}^{2} - \cdots - \lambda_{n-1}v_{n-1}^{2}. \\
    \end{aligned}
\right. \label{ODE_2}
\end{equation}
The solution of the above equation (\ref{ODE_2}) satisfies 
\[ v_{3}^{2} + \cdots + v_{n}^{2} = C, \]
which is contained in a bounded closed set. 
Thus, the solution of (\ref{ODE_1}) exists globally. 
Next, we assume that $(u_{1}(t_0), u_{2}(t_0)) \neq (0, 0)$ as initial values. 
Then the solution $(u_{1}, u_{2})$ cannot pass the origin $(0, 0)$. 
Setting 
\[ u_{1}(t) = r(t)\cos{\theta(t)}, \quad u_{2}(t) = r(t) \sin{\theta(t)}, \]
we see that $r(t)$ is positive and $\theta(t)$ is continuous. 
Then we compute that 
\[ \dot{r} = \alpha u_{n}r, \quad \dot{\theta} = \beta u_{n}. \]
by using (\ref{ODE_1}). It implies 
\[ \frac{\dot{r}}{r} = \frac{\alpha}{\beta}\dot{\theta}. \]
By integrating and 
setting $r_{0}=r(t_{0}), \, \theta_{0} = \theta(t_{0})$, 
we have 
\[ r = r_{0} \exp{\left\{\frac{\alpha}{\beta}
\left(\theta - \theta_{0}\right)\right\}}. \]
On the other hand, since we have 
\[ u_{3}^{2} + \cdots + u_{n}^{2} = C + u_{1}^{2} - u_{2}^{2} 
= C + r^{2}\cos{2\theta} 
= C + r_{0}^{2} \exp{\left\{\frac{2\alpha}{\beta}(\theta-\theta_{0})\right\}}\cos{2\theta} \]
from the conserved quantity, we must have 
\[ r_{0}^{2} \exp{\left\{\frac{2\alpha}{\beta}(\theta-\theta_{0})\right\}}\cos{2\theta} \geq -C. \]
Thus, $\theta(t)$ 
is bounded below because it is continuous. 
Since $r(t)$ is bounded, 
$u_{1}(t)$ is bounded. 
Consequently, the conserved quantity ensures that $u_{2}(t), \ldots, u_{n}(t)$ are also bounded. 
Hence, the solution of (\ref{ODE_1}) exists globally since it is contained in a bounded closed set. 
\end{proof}

\subsection{Closed timelike curve}

A Lorentzian manifold $(M, g_{M})$ is \textit{totally vicious} 
if every point of $M$ lies on a closed timelike curve (CTC).
We show that $\mathbb{R}^n$ equipped with the metric \eqref{main} is totally vicious. 
In particular, it is not globally hyperbolic. 

\begin{prop} \label{causality}
Let $\lambda_{3}, \ldots, \lambda_{n-1}\in\mathbb{R}$, $\alpha\le0$, $\beta > 0$, and $g_{\mathbb{R}^{n}}$ be the Lorentzian metric on $\mathbb{R}^{n}$ given by (\ref{main}). 
Then the homogeneous Lorentzian manifold 
$(\mathbb{R}^{n}, g_{\mathbb{R}^{n}})$ is totally vicious. 
\end{prop}

\begin{proof}
From the homogeneity, 
it suffices to show the existence of a closed timelike curve. 
We define a smooth and regular closed curve as 
\begin{align}
    c(t) = 
\frac{1}{\beta}\left(3
\sin{t}, -\sin{2t}, 0, \ldots, 0, \frac{\pi}{2}(1-\cos{t}) \right) 
\in \mathbb{R}^{n}.
\label{eq:c}
\end{align}
Then we compute that 
\begin{align}
&\beta^{2} g_{\mathbb{R}^{n}}(c'(t), c'(t)) 
= \exp{\left\{\frac{\pi|\alpha|}{\beta}(1-\cos{t} )\right\}}f(t) + \frac{\pi^{2}}{4}\sin^{2}{t}, \label{eq:bgcc} \\
&f(t):=\cos{(\pi\cos{t})}(9\cos^{2}{t} - 4\cos^{2}{2t}) 
+12\sin{(\pi\cos{t})}\cos{2t}\cos{t}. \label{eq:f}
\end{align}
By numerical calculation (see Figure~\ref{fig:f}), the maximal value of $f(t)$ satisfies 
\[ \max{f(t)} = -2.72\cdots < -\frac{\pi^2}{4} = -2.46\cdots . \]
Since the exponential factor in \eqref{eq:bgcc} is greater than or equal to 1, we obtain
\begin{align*}
\beta^{2} g_{\mathbb{R}^{n}}(c'(t), c'(t)) 
&< -\frac{\pi^{2}}{4} + \frac{\pi^{2}}{4} = 0.
\end{align*}
Thus, the velocity vector $c'(t)$ is timelike for any $t \in \mathbb{R}$. 
\end{proof}
\begin{figure}[htbp]
    \centering
    \includegraphics[width=0.7\linewidth]{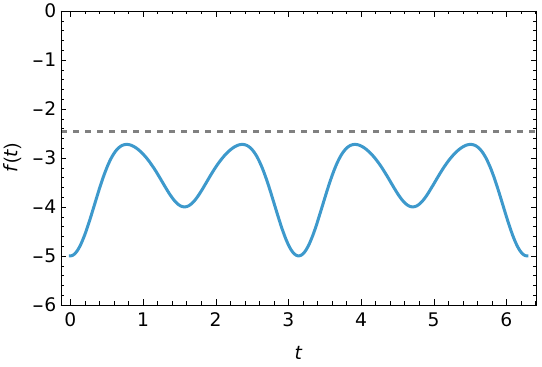}
    \caption{Numerical calculation of $f(t)$ defined in \eqref{eq:f}. The dashed horizontal line shows $-\pi^2/{4} = -2.46\cdots$.}
    \label{fig:f}
\end{figure}
\begin{figure}[htpb]
    \centering
    \includegraphics[width=0.6\linewidth]{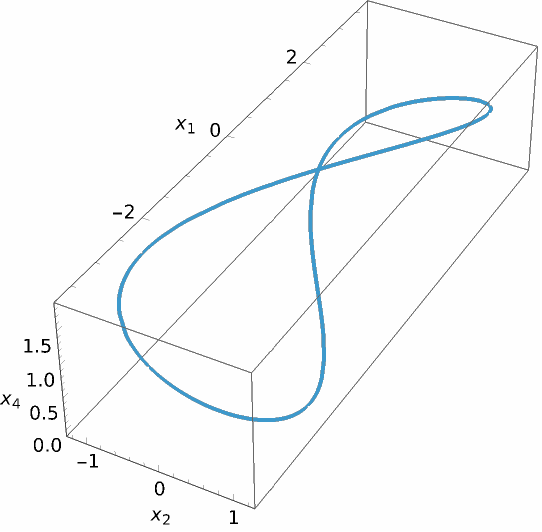}
    \caption{Closed timelike curve (CTC) defined in \eqref{eq:c} in the case $n=4,\ \beta = \sqrt{3}/2$.
    This is the newly found CTC of the Petrov solution \eqref{Petrov_sol}.}
    \label{fig:ctc}
\end{figure}
\begin{figure}[htpb]
  \begin{minipage}[b]{0.45\linewidth}
    \centering
    \includegraphics[scale=0.8]{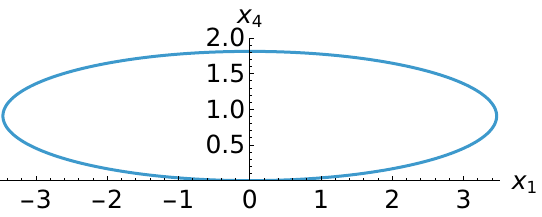}
  \end{minipage}%
  \hspace{0.03\columnwidth}
  \begin{minipage}[b]{0.45\linewidth}
    \centering
    \includegraphics[scale=0.8]{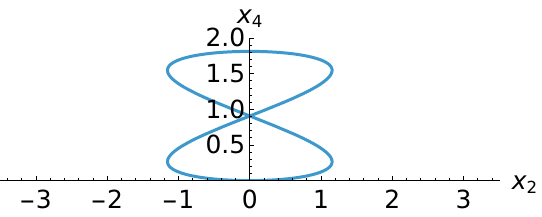}
  \end{minipage}
  \caption{Two-dimensional plots of the CTC in the $x_1$-$x_4$ (left) and $x_2$-$x_4$ (right) planes.}
  \label{fig:2d}
\end{figure}

For the Petrov solution \eqref{Petrov_sol}, the CTC can be obtained by substituting $\beta = \sqrt{3}/2$ in \eqref{eq:c}.
The curve is shown in Figures~\ref{fig:ctc} and \ref{fig:2d}.
It can be regarded as a three-dimensional Lissajous curve.
The existence of a CTC in the Petrov solution has been known only when a particular identification of a coordinate is assumed \cite{MR408733,MR2429726}.
That identification originates from interpreting the Petrov solution as the vacuum solution in the cylindrically symmetric case \cite{MR594564}.
The CTC constructed in this paper does not require such an identification of any coordinate.
It indicates that a CTC can exist for a matter distribution that is not cylindrically symmetric.

\begin{rem} 
We denote the set of Lorentzian metrics on $\mathbb{R}^{n}$ by $\mathrm{Lor}(\mathbb{R}^{n})$, 
and we define the subset of $\mathrm{Lor}(\mathbb{R}^{n})$ as 
\[ \mathcal{V}(\mathbb{R}^{n}) := \{ g_{\mathbb{R}^{n}} \in \mathrm{Lor}(\mathbb{R}^{n}) \mid g_{\mathbb{R}^{n}} \, \textrm{is Ricci-flat} \}. \]
This set could represent the moduli space of vacuum solutions on $\mathbb{R}^{n}$. 
Since the metric \eqref{main} reduces to the Minkowski metric 
\[ g_{\mathbb{R}^{n}} = -dx_{1}^{2} + dx_{2}^{2} + \cdots 
+ dx_{n-1}^{2} + dx_{n}^{2} \]
when $\beta = 0$, the Ricci-flat spacetime under consideration yields a continuous path in $\mathcal{V}(\mathbb{R}^{n})$ connecting a totally vicious spacetime to a globally hyperbolic one with respect to $C^{\infty}$ topology, but not with respect to Whitney fine $C^{\infty}$ topology \cite{MR2906563}. 
Hence, the extreme levels, globally hyperbolic and totally vicious, of the causal hierarchy \cite{MR2436235} are not separated even within $\mathcal{V}(\mathbb{R}^{n})$. 
\end{rem}

\section*{Acknowledgments}
The first author is supported by Waseda University Grants for Special Research Projects (Grant No.~2024C-770 and No.~2025E-043).

\appendix
\section{Computation of Ricci operator}
\label{Ric_op}
By using Proposition~\ref{Ric_curv}, we explicitly compute the components of the Ricci tensor for a Lorentzian almost abelian Lie algebra 
$(\mathfrak{g}, \langle \, , \, \rangle)$ with the associated matrix 
$A = [A^{i}_{j}]_{1 \leq i, j\leq n-1}$ with respect to 
$\{X_{a}\}_{1\leq a \leq n}$. 
We recall that the entries $A^{i}_{j}$ are defined by (\ref{eA}). 
Here we use the indices 
$1 \leq i, j, \ldots \leq n-1$ and $1 \leq a, b, \ldots \leq n$. 

\subsection{Case (a)}
For the metric $\langle \, , \, \rangle = \langle \, , \, \rangle_{\mathrm{a}}$, 
we see $\varepsilon_{1} = \ldots = \varepsilon_{n-1} = -\varepsilon_{n} = 1$, 
where $\varepsilon_{a} = \langle X_{a}, X_{a} \rangle$. 
Then we compute that 
\begin{flalign*}
\mathrm{Ric}(X_{i}, X_{j}) &= 
\frac{1}{2}\langle [X_{n}, X_{i}], [X_{n}, X_{j}] \rangle 
-\frac{1}{2}\sum_{k=1}^{n-1}\langle [X_{n}, X_{k}], X_{i} \rangle 
\langle [X_{n}, X_{k}], X_{j} \rangle& \\
&\qquad +\frac{\mathrm{tr}\, A}{2} ( \langle [X_{n}, X_{i}], X_{j} \rangle + 
\langle [X_{n}, X_{j}], X_{i} \rangle ) &\\
&= \frac{1}{2}\sum_{k=1}^{n-1}A^{k}_{i}A^{k}_{j} 
- \frac{1}{2}\sum_{k=1}^{n-1}A^{i}_{k}A^{j}_{k} 
+ \frac{\mathrm{tr}\, A}{2}\left( A^{i}_{j} + A^{j}_{i} \right) \\
&= \frac{1}{2}\left[ A^{T}A - AA^{T} + (\mathrm{tr}\, A)(A + A^{T}) \right]_{ij}, 
\end{flalign*}
\begin{flalign*}
\mathrm{Ric}(X_{n}, X_{j}) &= 
-\frac{1}{2}\sum_{k=1}^{n-1}\langle [X_{n}, X_{k}], X_{n} \rangle 
\langle [X_{n}, X_{k}], X_{j} \rangle +\frac{\mathrm{tr}\, A}{2} ( \langle [X_{n}, X_{n}], X_{j} \rangle + 
\langle [X_{n}, X_{j}], X_{n} \rangle ) &\\ 
&= 0,
\end{flalign*}
\begin{flalign*}
\mathrm{Ric}(X_{n}, X_{n}) &= 
-\frac{1}{2}\sum_{i=1}^{n-1}\langle [X_{n}, X_{i}], [X_{n}, X_{i}] \rangle 
-\frac{\mathrm{tr}\, A^{2}}{2} = -\frac{1}{2}\sum_{i,k=1}^{n-1}A_{i}^{k}A_{i}^{k} 
- \frac{\mathrm{tr}\, A^{2}}{2} &\\
&= -\frac{1}{2}\left( \mathrm{tr}\, (AA^{T}) + \mathrm{tr}\, A^{2} \right). 
\end{flalign*}
Thus, by using the decomposition $A = S + T$ in (\ref{split}),
we have 
\begin{equation*}
\mathrm{Ric} = 
-\frac{1}{2}
\begin{bmatrix}
    \, \, [A, A^{T}] - (\mathrm{tr}\, A)(A + A^{T}) & \bm{0} \, \\
    \, \, \bm{0}^{T} & \mathrm{tr}\, (AA^{T} + A^{2}) \, \\ 
\end{bmatrix}
= 
\begin{bmatrix}
    \, \, [S, T] + (\mathrm{tr}\, S)S & \bm{0} \, \\
    \, \, \bm{0}^{T} & -\mathrm{tr}\, S^{2} \, \\ 
\end{bmatrix}
. 
\end{equation*}

\subsection{Case (b)}
For the metric $\langle \, , \, \rangle = \langle \, , \, \rangle_{\mathrm{b}}$, 
we see $-\varepsilon_{1} = \varepsilon_{2} = \ldots = \varepsilon_{n} = 1$,  
Then we compute that 
\begin{flalign*}
\mathrm{Ric}(X_{i}, X_{j}) &= 
-\frac{1}{2}\langle [X_{n}, X_{i}], [X_{n}, X_{j}] \rangle 
+\frac{1}{2}\sum_{k=1}^{n-1}\langle [X_{n}, X_{k}], X_{i} \rangle 
\langle [X_{n}, X_{k}], X_{j} \rangle \varepsilon_{k} &\\
&\qquad - \frac{\mathrm{tr}\, A}{2} ( \langle [X_{n}, X_{i}], X_{j} \rangle + 
\langle [X_{n}, X_{j}], X_{i} \rangle ) \\
&= -\frac{1}{2}\sum_{k=1}^{n-1} \varepsilon_{k}A_{i}^{k}A_{j}^{k} + 
\frac{1}{2}\varepsilon_{i}\varepsilon_{j}\sum_{k=1}^{n-1}\varepsilon_{k}A^{i}_{k}A^{j}_{k} - \frac{\mathrm{tr}\, A}{2}(\varepsilon_{i}A^{i}_{j} + \varepsilon_{j}A^{j}_{i}) \\
&= -\frac{1}{2}\left[ A^{T}JA - JAJA^{T}J + (\mathrm{tr}\, A)(JA + A^{T}J) \right]_{ij},
\end{flalign*}
\begin{flalign*}
    \mathrm{Ric}(X_{n}, X_{j}) &= 
\frac{1}{2}\sum_{k=1}^{n-1}\langle [X_{n}, X_{k}], X_{n} \rangle 
\langle [X_{n}, X_{k}], X_{j} \rangle &\\
&\qquad -\frac{\mathrm{tr}\, A}{2} ( \langle [X_{n}, X_{n}], X_{j} \rangle + 
\langle [X_{n}, X_{j}], X_{n} \rangle ) = 0, 
\end{flalign*}
\begin{flalign*}
\mathrm{Ric}(X_{n}, X_{n}) &= 
-\frac{1}{2}\sum_{i=1}^{n-1}\langle [X_{n}, X_{i}], [X_{n}, X_{i}] \rangle \varepsilon_{i} -\frac{\mathrm{tr}\, A^{2}}{2} &\\
&= -\frac{1}{2}\sum_{i,k=1}^{n-1}A_{i}^{k}A_{i}^{k}\varepsilon_{i}\varepsilon_{k} 
- \frac{\mathrm{tr}\, A^{2}}{2} = 
-\frac{1}{2}\left( \mathrm{tr}\, (JAJA^{T} + A^{2}) \right). 
\end{flalign*}
Thus, by using the $J$-decomposition $A = S_{L} + T_{L}$ in (\ref{split}),
we have 
\begin{align*}
\mathrm{Ric} &= 
-\frac{1}{2}
\begin{bmatrix}
    \, \, A^{T}JA - JAJA^{T}J + (\mathrm{tr}\, A)(JA + A^{T}J) & \bm{0} \, \\
    \, \, \bm{0}^{T} & \mathrm{tr}\, (JAJA^{T} + A^{2}) \, \\ 
\end{bmatrix}
\\
&= 
\begin{bmatrix}
    \, -J([S_{L}, T_{L}] + (\mathrm{tr}\, S_{L})S_{L}) & \bm{0} \, \\
    \, \bm{0}^{T} & -\mathrm{tr}\,S_{L}^{2} \, \\
\end{bmatrix}
. 
\end{align*}

\subsection{Case (c)}
For the metric $\langle \, , \, \rangle = \langle \, , \, \rangle_{\mathrm{c}}$, 
we see $\varepsilon_{1} = \ldots = \varepsilon_{n-1}  = -\varepsilon_{n} = 1$, 
where $e_{h} = X_{h} \, (1 \leq h \leq n-2)$ and 
\[ 
e_{n-1} = \frac{1}{\sqrt{2}}(X_{n-1} + X_{n}), \quad 
e_{n} = \frac{1}{\sqrt{2}}(X_{n-1} - X_{n}). 
\] 
Then we compute that 
\begin{align*}
\mathrm{Ric}(X_{i}, X_{j}) &= 
-\frac{1}{4}\{\langle [X_{n-1} + X_{n}, X_{i}], [X_{n-1} + X_{n}, X_{j}] \rangle \\
&\qquad \qquad \qquad \qquad - \langle [X_{n-1} - X_{n}, X_{i}], [X_{n-1} - X_{n}, X_{j}] \rangle \} \\
&\qquad 
+\frac{1}{4}\sum_{k=1}^{n-2}\langle [X_{n-1} + X_{n}, X_{k}], X_{i} \rangle 
\langle [X_{n-1} + X_{n}, X_{k}], X_{j} \rangle \\
&\qquad \qquad 
-\frac{1}{4}\sum_{k=1}^{n-2}\langle [X_{n-1} - X_{n}, X_{k}], X_{i} \rangle 
\langle [X_{n-1} - X_{n}, X_{k}], X_{j} \rangle \\
&\qquad \qquad \qquad 
-\frac{1}{8}\langle [X_{n-1} + X_{n}, X_{n-1} - X_{n}], X_{i} \rangle 
\langle [X_{n-1} + X_{n}, X_{n-1} - X_{n}], X_{j} \rangle \\
&\qquad \qquad \qquad \qquad 
-\frac{\mathrm{tr}\, A}{2} ( \langle [X_{n-1}, X_{i}], X_{j} \rangle + 
\langle [X_{n-1}, X_{j}], X_{i} \rangle ) \\
&= -\frac{1}{2}\langle [X_{n}, X_{n-1}], X_{i} \rangle 
\langle [X_{n}, X_{n-1}], X_{j} \rangle \\
&= -\frac{1}{2}\sum_{k,l=1}^{n-1} A_{n-1}^{k}A_{n-1}^{l} \langle X_{k}, X_{i} \rangle 
\langle X_{l}, X_{j} \rangle \\ 
&= 
\begin{cases}
    \ -\dfrac{1}{2}A_{n-1}^{i}A_{n-1}^{j} \quad (1 \leq i, j \leq n-2), \\
    \ 0 \quad (\mathrm{otherwise}),
\end{cases}
\end{align*}
\begin{flalign*}
\mathrm{Ric}(X_{n}, X_{j}) &= 
\frac{1}{2} \langle [X_{n}, X_{n-1}], [X_{n}, X_{j}] \rangle 
-\frac{1}{2} \langle [X_{n}, X_{n-1}], X_{n} \rangle \langle [X_{n}, X_{n-1}], X_{j} \rangle &\\
&\qquad +\frac{\mathrm{tr}\, A}{2} \langle [X_{n}, X_{n-1}], X_{j} \rangle \\ 
&= \frac{1}{2}\left[ 
\sum_{k=1}^{n-2}A_{j}^{k}A_{n-1}^{k} 
+\sum_{l=1}^{n-1} \left\{ (\mathrm{tr}\, A) - A_{n-1}^{n-1} \right\} 
A_{n-1}^{l} \langle X_{l}, X_{j} \rangle \right] \\ 
&= 
\begin{cases}
    \ \dfrac{1}{2}\left[  
    \displaystyle \sum_{k=1}^{n-2}A_{j}^{k}A_{n-1}^{k} + 
    \left\{ (\mathrm{tr}\, A) - A_{n-1}^{n-1} \right\} A_{n-1}^{j} \right] 
    \quad (j \leq n-2), \\
    \ \dfrac{1}{2} \displaystyle \sum_{k=1}^{n-2}A_{n-1}^{k}A_{n-1}^{k}  \quad (j = n-1),
\end{cases} 
\end{flalign*}
\begin{flalign*}
\mathrm{Ric}(X_{n}, X_{n}) &= 
-\frac{1}{2}\sum_{i=1}^{n-2}\langle [X_{n}, X_{i}], [X_{n}, X_{i}] \rangle 
-\frac{1}{2}\langle [X_{n}, X_{n-1}], X_{n} \rangle 
\langle [X_{n}, X_{n-1}], X_{n} \rangle &\\ 
&\qquad -\frac{\mathrm{tr}\, A^{2}}{2} 
+ (\mathrm{tr}\, A)\langle [X_{n}, X_{n-1}], X_{n} \rangle \\
&= -\frac{1}{2}\left[ \sum_{i,j=1}^{n-2}A_{i}^{j}A_{i}^{j} + 
\left(\mathrm{tr}\, A^{2}\right) + \left(A_{n-1}^{n-1} \right)^{2} - 2(\mathrm{tr}\, A)A_{n-1}^{n-1} \right]. 
\end{flalign*}
Thus, by using the block form 
\[ A = 
\begin{bmatrix}
    A' & \bm{b} \, \\
    \bm{c}^{T} & d \, \\  
\end{bmatrix}  \]
of $A$ in (\ref{split}), and the decomposition $A' = S' + T'$, 
we have 
\begin{align*}
\mathrm{Ric} &= 
-\frac{1}{2}
\begin{bmatrix}
    \, \bm{b}\bm{b}^{T} & \bm{0} 
    & -\left[(\mathrm{tr}\, A')I_{n-2} + A'\right]\bm{b}  \, \\ 
    \, \bm{0}^{T} & 0 & -||\bm{b}||^{2} \, \\
    \, -\bm{b}^{T}\left[(\mathrm{tr}\, A')I_{n-2} + \left(A'\right)^{T}\right] 
    & -||\bm{b}||^{2} &  \mathrm{tr} \left[ A'(A')^{T} + (A')^{2} \right] - 2d(\mathrm{tr}\, A') 
    + 2\bm{b}^{T}\bm{c} \, \\
\end{bmatrix}
\\
&= 
\frac{1}{2} 
\begin{bmatrix}
    \, -\bm{b}\bm{b}^{T} & \bm{0} 
    & \left[(\mathrm{tr}\, S')I_{n-2} + S' + T'\right]\bm{b}  \, \\ 
    \, \bm{0}^{T} & 0 & ||\bm{b}||^{2} \, \\
    \, \bm{b}^{T}\left[(\mathrm{tr}\, S')I_{n-2} + S' - T'\right] 
    & ||\bm{b}||^{2} & 2d(\mathrm{tr}\, S') - 2\mathrm{tr}\, S'\,\!^{2}
    - 2\bm{b}^{T}\bm{c} \, \\
\end{bmatrix}
. 
\end{align*}

\section{Locally symmetric space} \label{loc_symm_sp}
In this section, we show that locally symmetric and Ricci-flat almost abelian Lie groups are flat.
The Lorentzian almost abelian Lie algebra 
$(\mathfrak{g}, \langle \, , \, \rangle)$ for $(G, ds^{2})$ is called 
{\it locally symmetric} if $\nabla R = 0$ holds, 
i.e., for any $X, Y, Z \in \mathfrak{g}$, 
\[ (\nabla_{Z}R)(X,Y) := L_{Z}R(X,Y) - R(L_{Z}X,Y) - 
R(X,L_{Z}Y) - R(X,Y) L_{Z} = 0. \]
Then the corresponding Lorentzian manifold 
$(G, ds^{2})$ is locally symmetric. 
If it is flat, then it is locally symmetric by definition. 

\begin{lem} \label{loc_symm}
For the Lorentzian almost abelian Lie algebra 
$(\mathfrak{g}, \langle \, , \, \rangle)$, 
it is Ricci-flat and locally symmetric if and only if it is flat. 
\end{lem}

\begin{proof}
There is nothing to prove for case (a). 
For case (b), let $S_{L}$ satisfy the Ricci-flat conditions. 
From $(\nabla_{X_{i}}R)(X_{n}, X_{j}) = 0 \ (1 \leq i, j \leq n-1)$, 
we have $S_{L}^{3} = O$. 
For each case of Proposition~\ref{normal_form}, 
when $S_{L}$ is in case (i), 
$S_{L}^{2} \neq O, \, S_{L}^{3} \neq O$ hold, 
that is, it is neither flat nor locally symmetric. 
When $S_{L}$ is in case (ii), it is flat. 
When $S_{L}$ is in case (iii), 
we see that $S_{L}^{2} \neq O$. 
Thus, it is Ricci-flat but not flat. 
Moreover, we see that $(\nabla_{X_{1}}R)(X_{2}, X_{3}) \neq 0$, 
that is, it is Ricci-flat but not locally symmetric. 
For case (c), there is nothing to check when $d = 0$. 
Let $d = 1$. From $(\nabla_{X_{n}}R)(X_{n}, X_{n-1}) = 0$, 
we have 
\[ 2W = [W, T'], \]
where $W := S'\,\!^{2} - S' + [S', T']$. 
Since $W$ is symmetric and $T'$ is skew-symmetric, 
it follows that $W = O$. 
In particular, since 
$S'\,\!^{2} - S' = O, \, [S', T'] = O$, it is flat. 
\end{proof}

\section{Decomposability} \label{indecom} 

A Lie algebra $\mathfrak{g}$ is called \textit{decomposable} 
if there exist non-trivial Lie subalgebras 
$\mathfrak{g}_{1}, \, \mathfrak{g}_{2} \subset \mathfrak{g}$ such that 
it decomposes into a direct sum $\mathfrak{g} = \mathfrak{g}_{1} \oplus \mathfrak{g}_{2}$. 
Otherwise, $\mathfrak{g}$ is called \textit{indecomposable}. 
A Lie group is defined to be decomposable 
if its Lie algebra is decomposable.

\begin{prop} \label{decom}
For the solution $\mathbb{P}(\lambda_{3}, \ldots, \lambda_{n-1})$, 
it is decomposable if and only if there exists $3 \leq i \leq n-1$ such that $\lambda_{i} = 0$. 
Furthermore, if the number of zeros among $\lambda_i$ is $m$, by reordering $\{\lambda_i\}_{3\le i\le n-1}$, we obtain 
\begin{align}
    \mathbb{P}(\lambda_{3}, \ldots, \lambda_{n-m-1},0,\dots,0)=\mathbb{P}(\lambda_{3}, \ldots, \lambda_{n-m-1})\times \mathbb{E}^m, 
\end{align}
where $\mathbb{E}^m$ denotes a flat Euclidean $m$-space.
\end{prop}

\begin{proof}
Assume that there exists $3 \leq i \leq n-1$ such that $\lambda_{i} = 0$. 
We define two subspaces of $\mathfrak{g}$ as
\[ \mathfrak{g}_{1} := \mathrm{span}_{\mathbb{R}} \{ X_{1}, \ldots, X_{i-1}, X_{i+1}, \ldots, X_{n} \}, \quad 
\mathfrak{g}_{2} := \mathbb{R} X_{i}. \]
Since $\mathfrak{g} = \mathfrak{g}_{1} + \mathfrak{g}_{2}$ and $\mathfrak{g}_{1} \cap \mathfrak{g}_{2} = \{ 0 \}$ hold, 
we have 
\[ [\mathfrak{g}_{1}, \mathfrak{g}_{1}] \subset \mathfrak{g}_{1}, \quad 
[\mathfrak{g}_{1}, \mathfrak{g}_{2}] = \{ 0 \}.\]
Hence, $\mathfrak{g} = \mathfrak{g}_{1} \oplus \mathfrak{g}_{2}$ holds. 
Conversely, we give a proof by contradiction, that is, 
assume that $\lambda_{i} \neq 0$ for any $i = 3, \ldots, n-1$. 
If $\mathfrak{g}$ is decomposable, 
then $\mathfrak{g} = \mathfrak{g}' \oplus \mathfrak{a}'$, 
where $\mathfrak{g}'$ is almost abelian and $\mathfrak{a}'$ is abelian (see \cite[Lemma~1]{MR4462443}). 
Therefore, $[\mathfrak{a}', \mathfrak{g}] = [\mathfrak{a}', \mathfrak{g}'] = \{ 0 \}$. 
If $\mathfrak{a}' \subset \mathfrak{a}$ holds, 
then for $Y \in \mathfrak{a}' \setminus \{ 0 \}$, we can express 
\[ Y = \sum_{i = 1}^{n-1} \mu_{i} X_{i} \in \mathfrak{a} = \mathrm{span}_{\mathbb{R}} \{ X_{1}, \ldots, X_{n-1} \} \]
and there exists $1 \leq j \leq n-1$ such that $\mu_{j} \neq 0$. 
Then for $X_{n} \in \mathfrak{g}$ we have 
\[ [X_{n}, Y] = \mu_{1}(\alpha X_{1} + \beta X_{2}) + \mu_{2}(-\beta X_{1} + \alpha X_{2}) 
+ \mu_{3}\lambda_{3}X_{3} + \cdots + \mu_{n-1}\lambda_{n-1}X_{n-1} 
\neq 0. \]
Thus $[\mathfrak{a}', \mathfrak{g}] \neq \{0\}$. 
If $\mathfrak{a}' \subset \mathfrak{a}$ does not hold, 
then there exist $Z \in \mathfrak{a}', \, W \in \mathfrak{a}$ and $\rho \neq 0$ 
such that $Z = W + \rho X_{n} \in \mathfrak{g}$. 
Then we have 
\[ [Z, X_{3}] = \rho \lambda_{3} X_{3} \neq 0. \]
Thus, $[\mathfrak{a}', \mathfrak{g}] \neq \{0\}$,
which implies that $\mathfrak{g}$ is indecomposable. 
Furthermore, by setting 
\[ m := \# \{i \in \{3, \ldots, n-1 \} \mid \lambda_{i} = 0\} \] 
and reordering $\{\lambda_i\}_{3\le i\le n-1}$, we find that the metric~(\ref{main}) contains the part $dx_{n-m}^{2} + dx_{n-m+1}^{2} + \cdots + dx_{n-1}^{2}$. 
Then $(\mathbb{R}^{m}, \, dx_{n-m}^{2} + \cdots + dx_{n-1}^{2})$ is a flat Euclidean $m$-space, 
which proves the latter claim. 
\end{proof} 


\bibliographystyle{elsarticle-num}
\bibliography{ref}

@article {MR2995204,
    AUTHOR = {Freibert, Marco},
     TITLE = {Cocalibrated structures on {L}ie algebras with a codimension
              one {A}belian ideal},
   JOURNAL = {Ann. Global Anal. Geom.},
  FJOURNAL = {Annals of Global Analysis and Geometry},
    VOLUME = {42},
      YEAR = {2012},
    NUMBER = {4},
     PAGES = {537--563},
      ISSN = {0232-704X,1572-9060},
   MRCLASS = {53C25 (53C15 53C30)},
  MRNUMBER = {2995204},
MRREVIEWER = {Hamid\ Reza\ Salimi Moghaddam},
       DOI = {10.1007/s10455-012-9326-0},
       URL = {https://doi.org/10.1007/s10455-012-9326-0},
}

@article {MR408733,
    AUTHOR = {Tipler, Frank J.},
     TITLE = {Rotating cylinders and the possibility of global causality
              violation},
   JOURNAL = {Phys. Rev. D (3)},
  FJOURNAL = {Physical Review. D. Particles and Fields. Third Series},
    VOLUME = {9},
      YEAR = {1974},
     PAGES = {2203--2206},
      ISSN = {0556-2821},
   MRCLASS = {83.53},
  MRNUMBER = {408733},
       DOI = {10.1103/PhysRevD.9.2203},
       URL = {https://doi.org/10.1103/PhysRevD.9.2203},
}

@article {MR2429726,
    AUTHOR = {Gibbons, Gary W. and Gielen, Steffen},
     TITLE = {The {P}etrov and {K}aigorodov-{O}zsv\'ath solutions: spacetime
              as a group manifold},
   JOURNAL = {Classical Quantum Gravity},
  FJOURNAL = {Classical and Quantum Gravity},
    VOLUME = {25},
      YEAR = {2008},
    NUMBER = {16},
     PAGES = {165009, 23},
      ISSN = {0264-9381,1361-6382},
   MRCLASS = {83C15 (83C20)},
  MRNUMBER = {2429726},
MRREVIEWER = {Georgios\ O.\ Papadopoulos},
       DOI = {10.1088/0264-9381/25/16/165009},
       URL = {https://doi.org/10.1088/0264-9381/25/16/165009},
}

@article {MR4752280,
    AUTHOR = {An, Huihui and Deng, Shaoqiang and Yan, Zaili},
     TITLE = {Nilpotency of the {R}icci operator of pseudo-{R}iemannian
              solvmanifolds},
   JOURNAL = {Bull. Korean Math. Soc.},
  FJOURNAL = {Bulletin of the Korean Mathematical Society},
    VOLUME = {61},
      YEAR = {2024},
    NUMBER = {3},
     PAGES = {867--873},
      ISSN = {1015-8634,2234-3016},
   MRCLASS = {53C50 (53C24 53C30)},
  MRNUMBER = {4752280},
MRREVIEWER = {S.\ M. B. Kashani},
       DOI = {10.4134/BKMS.b230497},
       URL = {https://doi.org/10.4134/BKMS.b230497},
}

@article {MR2906563,
    AUTHOR = {Benavides Navarro, J. J. and Minguzzi, E.},
     TITLE = {Global hyperbolicity is stable in the interval topology},
   JOURNAL = {J. Math. Phys.},
  FJOURNAL = {Journal of Mathematical Physics},
    VOLUME = {52},
      YEAR = {2011},
    NUMBER = {11},
     PAGES = {112504, 8},
      ISSN = {0022-2488,1089-7658},
   MRCLASS = {53C50 (83C75)},
  MRNUMBER = {2906563},
MRREVIEWER = {Robert\ J.\ Low},
       DOI = {10.1063/1.3660684},
       URL = {https://doi.org/10.1063/1.3660684},
}

@incollection {MR2436235,
    AUTHOR = {Minguzzi, Ettore and S\'anchez, Miguel},
     TITLE = {The causal hierarchy of spacetimes},
 BOOKTITLE = {Recent developments in pseudo-{R}iemannian geometry},
    SERIES = {ESI Lect. Math. Phys.},
     PAGES = {299--358},
 PUBLISHER = {Eur. Math. Soc., Z\"urich},
      YEAR = {2008},
      ISBN = {978-3-03719-051-7},
   MRCLASS = {53C50 (53C80 83C75)},
  MRNUMBER = {2436235},
MRREVIEWER = {Jos\'e\ Luis\ Flores},
       DOI = {10.4171/051-1/9},
       URL = {https://doi.org/10.4171/051-1/9},
}

@article {MR2864471,
    AUTHOR = {Ha, Ku Yong and Lee, Jong Bum},
     TITLE = {The isometry groups of simply connected 3-dimensional
              unimodular {L}ie groups},
   JOURNAL = {J. Geom. Phys.},
  FJOURNAL = {Journal of Geometry and Physics},
    VOLUME = {62},
      YEAR = {2012},
    NUMBER = {2},
     PAGES = {189--203},
      ISSN = {0393-0440,1879-1662},
   MRCLASS = {53C30 (22E15)},
  MRNUMBER = {2864471},
MRREVIEWER = {Claudio\ Gorodski},
       DOI = {10.1016/j.geomphys.2011.10.011},
       URL = {https://doi.org/10.1016/j.geomphys.2011.10.011},
}

@article {MR4528847,
    AUTHOR = {Boucetta, Mohamed and Chakkar, Abdelmounaim},
     TITLE = {The isometry groups of {L}orentzian three-dimensional
              unimodular simply connected {L}ie groups},
   JOURNAL = {Rev. Un. Mat. Argentina},
  FJOURNAL = {Revista de la Uni\'on Matem\'atica Argentina},
    VOLUME = {63},
      YEAR = {2022},
    NUMBER = {2},
     PAGES = {353--378},
      ISSN = {0041-6932,1669-9637},
   MRCLASS = {22E15 (22E60 53C50)},
  MRNUMBER = {4528847},
MRREVIEWER = {Gabriela\ P.\ Ovando},
       DOI = {10.33044/revuma.2021},
       URL = {https://doi.org/10.33044/revuma.2021},
}

@article {MR594564,
    AUTHOR = {Bonnor, W. B.},
     TITLE = {A source for {P}etrov's homogeneous vacuum space-time},
   JOURNAL = {Phys. Lett. A},
  FJOURNAL = {Physics Letters. A},
    VOLUME = {75},
      YEAR = {1979/80},
    NUMBER = {1-2},
     PAGES = {25--26},
      ISSN = {0375-9601,1873-2429},
   MRCLASS = {83C15 (53B30)},
  MRNUMBER = {594564},
       DOI = {10.1016/0375-9601(79)90264-0},
       URL = {https://doi.org/10.1016/0375-9601(79)90264-0},
}

@article {MR31841,
    AUTHOR = {G\"odel, Kurt},
     TITLE = {An example of a new type of cosmological solutions of
              {E}instein's field equations of gravitation},
   JOURNAL = {Rev. Modern Physics},
  FJOURNAL = {Rev. Modern Physics},
    VOLUME = {21},
      YEAR = {1949},
     PAGES = {447--450},
   MRCLASS = {83.0X},
  MRNUMBER = {31841},
MRREVIEWER = {M.\ Wyman},
       DOI = {10.1103/revmodphys.21.447},
       URL = {https://doi.org/10.1103/revmodphys.21.447},
}

@article {MR4533011,
    AUTHOR = {Almora Rios, Marcelo and Avetisyan, Zhirayr and Berlow,
              Katalin and Martin, Isaac and Rakholia, Gautam and Yang,
              Kelley and Zhang, Hanwen and Zhao, Zishuo},
     TITLE = {Almost {A}belian {L}ie groups, subgroups and quotients},
   JOURNAL = {J. Math. Sci. (N.Y.)},
  FJOURNAL = {Journal of Mathematical Sciences (New York)},
    VOLUME = {266},
      YEAR = {2022},
    NUMBER = {1},
     PAGES = {42--65},
      ISSN = {1072-3374,1573-8795},
   MRCLASS = {22E25},
  MRNUMBER = {4533011},
MRREVIEWER = {Silvina\ Mabel\ Campos},
       DOI = {10.1007/s10958-022-05872-2},
       URL = {https://doi.org/10.1007/s10958-022-05872-2},
}

@article {MR195916,
    AUTHOR = {Kolman, Bernard},
     TITLE = {Semi-modular {L}ie algebras},
   JOURNAL = {J. Sci. Hiroshima Univ. Ser. A-I Math.},
  FJOURNAL = {Journal of Science of the Hiroshima University. Series A-I.
              Mathematics},
    VOLUME = {29},
      YEAR = {1965},
     PAGES = {149--163},
      ISSN = {0386-3026},
   MRCLASS = {17.30},
  MRNUMBER = {195916},
MRREVIEWER = {S.\ T\^og\^o},
}

@article {MR4609835,
    AUTHOR = {Yan, Zaili and Deng, Shaoqiang},
     TITLE = {Left invariant {R}icci flat metrics on {L}ie groups},
   JOURNAL = {Forum Math.},
  FJOURNAL = {Forum Mathematicum},
    VOLUME = {35},
      YEAR = {2023},
    NUMBER = {4},
     PAGES = {913--923},
      ISSN = {0933-7741,1435-5337},
   MRCLASS = {22E15 (22E25 53C25)},
  MRNUMBER = {4609835},
       DOI = {10.1515/forum-2022-0102},
       URL = {https://doi.org/10.1515/forum-2022-0102},
}

@article {MR3223494,
    AUTHOR = {Calvaruso, Giovanni and Zaeim, Amirhesam},
     TITLE = {Four-dimensional homogeneous {L}orentzian manifolds},
   JOURNAL = {Monatsh. Math.},
  FJOURNAL = {Monatshefte f\"{u}r Mathematik},
    VOLUME = {174},
      YEAR = {2014},
    NUMBER = {3},
     PAGES = {377--402},
      ISSN = {0026-9255,1436-5081},
   MRCLASS = {53C50 (53C30)},
  MRNUMBER = {3223494},
MRREVIEWER = {S.\ M. B. Kashani},
       DOI = {10.1007/s00605-013-0588-9},
       URL = {https://doi.org/10.1007/s00605-013-0588-9},
}

@book {MR0524082,
    AUTHOR = {Ryan, Jr., Michael P. and Shepley, Lawrence C.},
     TITLE = {Homogeneous relativistic cosmologies},
    SERIES = {Princeton Series in Physics},
 PUBLISHER = {Princeton University Press, Princeton, NJ},
      YEAR = {1975},
     PAGES = {xv+320},
   MRCLASS = {83.53},
  MRNUMBER = {524082},
}

@article {MR0475635,
    AUTHOR = {Hiromoto, Robert E. and Ozsv\'{a}th, Istv\'{a}n},
     TITLE = {On homogeneous solutions of {E}instein's field equations},
   JOURNAL = {Gen. Relativity Gravitation},
  FJOURNAL = {General Relativity and Gravitation},
    VOLUME = {9},
      YEAR = {1978},
    NUMBER = {4},
     PAGES = {299--327},
      ISSN = {0001-7701,1572-9532},
   MRCLASS = {83.53},
  MRNUMBER = {475635},
MRREVIEWER = {John\ Madore},
       DOI = {10.1080/07468342.2003.11922025},
       URL = {https://doi.org/10.1080/07468342.2003.11922025},
}

@article {MR1430782,
    AUTHOR = {Guediri, Mohammed},
     TITLE = {On completeness of left-invariant {L}orentz metrics on
              solvable {L}ie groups},
   JOURNAL = {Rev. Mat. Univ. Complut. Madrid},
  FJOURNAL = {Revista Matem\'atica de la Universidad Complutense de Madrid},
    VOLUME = {9},
      YEAR = {1996},
    NUMBER = {2},
     PAGES = {337--350},
      ISSN = {0214-3577},
   MRCLASS = {53C50 (22E25)},
  MRNUMBER = {1430782},
MRREVIEWER = {Maura\ B.\ Mast},
}

@incollection {MR0164700,
    AUTHOR = {Petrov, A. Z.},
     TITLE = {Gravitational field geometry as the geometry of automorphisms},
 BOOKTITLE = {Recent developments in general relativity},
     PAGES = {379--386},
 PUBLISHER = {Pergamon, Oxford; PWN---Polish Scientific Publishers, Warsaw},
      YEAR = {1962},
   MRCLASS = {83.22},
  MRNUMBER = {0164700},
MRREVIEWER = {F. A. E. Pirani},
}

@article {MR4859975,
    AUTHOR = {Sato, Yuichiro and Tsuyuki, Takanao},
     TITLE = {Spatially homogeneous solutions of vacuum {E}instein equations
              in general dimensions},
   JOURNAL = {J. Math. Phys.},
  FJOURNAL = {Journal of Mathematical Physics},
    VOLUME = {66},
      YEAR = {2025},
    NUMBER = {2},
     PAGES = {Paper No. 022501, 14},
      ISSN = {0022-2488,1089-7658},
   MRCLASS = {83E15 (53C80 83C20)},
  MRNUMBER = {4859975},
       DOI = {10.1063/5.0201573},
       URL = {https://doi.org/10.1063/5.0201573},
}

@article {MR402650,
    AUTHOR = {Alekseevski\u{i}, D. V. and  Kimel'fel'd, B. N.},
     TITLE = {Structure of homogeneous {R}iemannian spaces with zero {R}icci
              curvature},
   JOURNAL = {Funkcional. Anal. i Prilo\v zen.},
  FJOURNAL = {Akademija Nauk SSSR. Funkcional\cprime nyi Analiz i ego
              Prilo\v zenija},
    VOLUME = {9},
      YEAR = {1975},
    NUMBER = {2},
     PAGES = {5--11},
      ISSN = {0374-1990},
   MRCLASS = {53C30},
  MRNUMBER = {402650},
MRREVIEWER = {S.\ Kaneyuki},
}

@article {MR4101481,
    AUTHOR = {Boucetta, Mohamed and Tibssirte, Oumaima},
     TITLE = {On {E}instein {L}orentzian nilpotent {L}ie groups},
   JOURNAL = {J. Pure Appl. Algebra},
  FJOURNAL = {Journal of Pure and Applied Algebra},
    VOLUME = {224},
      YEAR = {2020},
    NUMBER = {12},
     PAGES = {106443, 22},
      ISSN = {0022-4049,1873-1376},
   MRCLASS = {22E25 (53C25 53C50)},
  MRNUMBER = {4101481},
MRREVIEWER = {Walter\ D.\ Freyn},
       DOI = {10.1016/j.jpaa.2020.106443},
       URL = {https://doi.org/10.1016/j.jpaa.2020.106443},
}

@article {MR3954004,
    AUTHOR = {Conti, Diego and Rossi, Federico A.},
     TITLE = {Ricci-flat and {E}instein pseudoriemannian nilmanifolds},
   JOURNAL = {Complex Manifolds},
  FJOURNAL = {Complex Manifolds},
    VOLUME = {6},
      YEAR = {2019},
    NUMBER = {1},
     PAGES = {170--193},
      ISSN = {2300-7443},
   MRCLASS = {53C25 (22E25 53C30 53C50)},
  MRNUMBER = {3954004},
MRREVIEWER = {Meera\ G.\ Mainkar},
       DOI = {10.1515/coma-2019-0010},
       URL = {https://doi.org/10.1515/coma-2019-0010},
}

@article {MR4319913,
    AUTHOR = {Xiang, Yujian and Yan, Zaili},
     TITLE = {Existence of left invariant {R}icci flat metrics on nilpotent
              {L}ie groups},
   JOURNAL = {Arch. Math. (Basel)},
  FJOURNAL = {Archiv der Mathematik},
    VOLUME = {117},
      YEAR = {2021},
    NUMBER = {5},
     PAGES = {569--578},
      ISSN = {0003-889X,1420-8938},
   MRCLASS = {53C25 (22E25 53C50)},
  MRNUMBER = {4319913},
MRREVIEWER = {Zhiqi\ Chen},
       DOI = {10.1007/s00013-021-01645-6},
       URL = {https://doi.org/10.1007/s00013-021-01645-6},
}

@article {MR3049631,
    AUTHOR = {Freibert, Marco},
     TITLE = {Cocalibrated {$G_2$}-structures on products of four- and
              three-dimensional {L}ie groups},
   JOURNAL = {Differential Geom. Appl.},
  FJOURNAL = {Differential Geometry and its Applications},
    VOLUME = {31},
      YEAR = {2013},
    NUMBER = {3},
     PAGES = {349--373},
      ISSN = {0926-2245,1872-6984},
   MRCLASS = {53C10 (53C15 53C30)},
  MRNUMBER = {3049631},
MRREVIEWER = {Andreas\ Arvanitoyeorgos},
       DOI = {10.1016/j.difgeo.2013.02.002},
       URL = {https://doi.org/10.1016/j.difgeo.2013.02.002},
}

@article {MR3772583,
    AUTHOR = {Console, S. and Macr\`i, M.},
     TITLE = {Lattices, cohomology and models of 6-dimensional almost
              abelian solvmanifolds},
   JOURNAL = {Rend. Semin. Mat. Univ. Politec. Torino},
  FJOURNAL = {Rendiconti del Seminario Matematico. Universit\`a{} e
              Politecnico Torino},
    VOLUME = {74},
      YEAR = {2016},
    NUMBER = {1},
     PAGES = {95--119},
      ISSN = {0373-1243,2704-999X},
   MRCLASS = {53C30 (22E40)},
  MRNUMBER = {3772583},
MRREVIEWER = {Salah\ Mehdi},
}

@article {MR3263659,
    AUTHOR = {Guediri, Mohammed and Bin-Asfour, Mona},
     TITLE = {Ricci-flat left-invariant {L}orentzian metrics on 2-step
              nilpotent {L}ie groups},
   JOURNAL = {Arch. Math. (Brno)},
  FJOURNAL = {Universitatis Masarykianae Brunensis. Facultas Scientiarum
              Naturalium. Archivum Mathematicum},
    VOLUME = {50},
      YEAR = {2014},
    NUMBER = {3},
     PAGES = {171--192},
      ISSN = {0044-8753},
   MRCLASS = {53C50 (22E25 53C25)},
  MRNUMBER = {3263659},
MRREVIEWER = {Oliver Baues},
       DOI = {10.5817/AM2014-3-171},
       URL = {https://doi.org/10.5817/AM2014-3-171},
}

@article {MR1334520,
    AUTHOR = {Witten, Edward},
     TITLE = {String theory dynamics in various dimensions},
   JOURNAL = {Nuclear Phys. B},
  FJOURNAL = {Nuclear Physics. B. Theoretical, Phenomenological, and
              Experimental High Energy Physics. Quantum Field Theory and
              Statistical Systems},
    VOLUME = {443},
      YEAR = {1995},
    NUMBER = {1-2},
     PAGES = {85--126},
      ISSN = {0550-3213,1873-1562},
   MRCLASS = {81T30 (83E30 83E50)},
  MRNUMBER = {1334520},
MRREVIEWER = {Tristan\ H\"ubsch},
       DOI = {10.1016/0550-3213(95)00158-O},
       URL = {https://doi.org/10.1016/0550-3213(95)00158-O},
}

@article {MR4462443,
    AUTHOR = {Avetisyan, Zhirayr},
     TITLE = {The structure of almost {A}belian {L}ie algebras},
   JOURNAL = {Internat. J. Math.},
  FJOURNAL = {International Journal of Mathematics},
    VOLUME = {33},
      YEAR = {2022},
    NUMBER = {8},
     PAGES = {Paper No. 2250057, 26},
      ISSN = {0129-167X,1793-6519},
   MRCLASS = {17B30},
  MRNUMBER = {4462443},
MRREVIEWER = {Vesselin\ Drensky},
       DOI = {10.1142/S0129167X22500574},
       URL = {https://doi.org/10.1142/S0129167X22500574},
}

@article {MR1332936,
    AUTHOR = {Singer, David A. and Steinberg, Daniel H.},
     TITLE = {Normal forms in {L}orentzian spaces},
   JOURNAL = {Nova J. Algebra Geom.},
  FJOURNAL = {Nova Journal of Algebra and Geometry},
    VOLUME = {3},
      YEAR = {1994},
    NUMBER = {1},
     PAGES = {1--9},
      ISSN = {1060-9881},
   MRCLASS = {51M10 (15A21)},
  MRNUMBER = {1332936},
MRREVIEWER = {Robert\ R.\ Miner},
}

@article {MR2254049,
    AUTHOR = {Leistner, Thomas},
     TITLE = {Conformal holonomy of {C}-spaces, {R}icci-flat, and
              {L}orentzian manifolds},
   JOURNAL = {Differential Geom. Appl.},
  FJOURNAL = {Differential Geometry and its Applications},
    VOLUME = {24},
      YEAR = {2006},
    NUMBER = {5},
     PAGES = {458--478},
      ISSN = {0926-2245,1872-6984},
   MRCLASS = {53C29 (53A30 53C50)},
  MRNUMBER = {2254049},
MRREVIEWER = {Andreas\ Cap},
       DOI = {10.1016/j.difgeo.2006.04.008},
       URL = {https://doi.org/10.1016/j.difgeo.2006.04.008},
}

@article {MR3539491,
    AUTHOR = {Globke, Wolfgang and Leistner, Thomas},
     TITLE = {Locally homogeneous pp-waves},
   JOURNAL = {J. Geom. Phys.},
  FJOURNAL = {Journal of Geometry and Physics},
    VOLUME = {108},
      YEAR = {2016},
     PAGES = {83--101},
      ISSN = {0393-0440,1879-1662},
   MRCLASS = {53C50 (53B30 53C29 53C30 83C20)},
  MRNUMBER = {3539491},
MRREVIEWER = {Sandra\ Gavino-Fern\'andez},
       DOI = {10.1016/j.geomphys.2016.06.013},
       URL = {https://doi.org/10.1016/j.geomphys.2016.06.013},
}

@article {Ref,
    AUTHOR = {Navas, S. and et al.},
     TITLE = {Review of particle physics},
   JOURNAL = {Phys. Rev. D},
  FJOURNAL = {Physical Review D},
    VOLUME = {110},
      YEAR = {2024},
    NUMBER = {3},
     PAGES = {Paper No. 030001},
       DOI = {10.1103/PhysRevD.110.030001},
}

@article{ParticleDataGroup:2024cfk,
    author = "Navas, S. and others",
    collaboration = "Particle Data Group",
    title = "{Review of particle physics}",
    doi = "10.1103/PhysRevD.110.030001",
    journal = "Phys. Rev. D",
    volume = "110",
    number = "3",
    pages = "030001",
    year = "2024"
}

@book {MR3155203,
    AUTHOR = {Green, Michael B. and Schwarz, John H. and Witten, Edward},
     TITLE = {Superstring theory. {V}ol. 1. {I}ntroduction},
   EDITION = {anniversary},
 PUBLISHER = {Cambridge University Press, Cambridge},
      YEAR = {2012},
     PAGES = {x+470},
      ISBN = {978-1-107-02911-8},
   MRCLASS = {81-02 (81T30 83-02 83E30)},
  MRNUMBER = {3155203},
MRREVIEWER = {Douglas\ J.\ Smith},
}

@book {MR2003646,
    AUTHOR = {Stephani, Hans and Kramer, Dietrich and MacCallum, Malcolm and
              Hoenselaers, Cornelius and Herlt, Eduard},
     TITLE = {Exact solutions of {E}instein's field equations},
    SERIES = {Cambridge Monographs on Mathematical Physics},
   EDITION = {Second},
 PUBLISHER = {Cambridge University Press, Cambridge},
      YEAR = {2003},
     PAGES = {xxx+701},
      ISBN = {0-521-46136-7},
   MRCLASS = {83C15 (83-02 83C20)},
  MRNUMBER = {2003646},
MRREVIEWER = {J. B. Griffiths},
       DOI = {10.1017/CBO9780511535185},
       URL = {https://doi.org/10.1017/CBO9780511535185},
}

@book {MR2371700,
    AUTHOR = {Besse, Arthur L.},
     TITLE = {Einstein manifolds},
    SERIES = {Classics in Mathematics},
      NOTE = {Reprint of the 1987 edition},
 PUBLISHER = {Springer-Verlag, Berlin},
      YEAR = {2008},
     PAGES = {xii+516},
      ISBN = {978-3-540-74120-6},
   MRCLASS = {53C25 (53-02)},
  MRNUMBER = {2371700},
}

\end{document}